\documentclass{amsart}
\title[Noncommutative rational functions]{Noncommutative rational functions,
their difference-differential calculus and realizations}

\author[D.~S.~Kaliuzhnyi-Verbovetskyi]{Dmitry S.
Kaliuzhnyi-Verbovetskyi}
\address{Department of Mathematics \\
Drexel University\\
3141 Chestnut Str.\\
 Philadelphia, PA, 19104}
\email{dmitryk@math.drexel.edu}
\author[V. Vinnikov]{Victor Vinnikov}
\address{Department of Mathematics \\
Ben-Gurion University of the Negev\\
 Beer-Sheva, Israel, 84105}
\email{vinnikov@math.bgu.ac.il}\thanks {Part of the research
described in this paper was carried during the first author's
visit to Ben-Gurion University in December 2009 that was partially
supported by the Center for Advanced Studies in Mathematics. The
revised version was prepared during the stay of the authors in May
2010 at the Mathematisches Forschungsinstitut Oberwolfach under
the program Research in Pairs. The first author was also supported
by the NSF grant DMS 0901628. The research of the second author
was partially supported by the Israel Science Foundation.}
\date{}
\usepackage{amssymb}

\newcommand{\ring}{\ensuremath{\mathcal{R}}}
\newcommand{\field}{\ensuremath{\mathbb{K}}}
\newcommand{\skfield}{\ensuremath{\mathbf{K}}}
\newcommand{\tfield}{\ensuremath{\widetilde{\mathbb{K}}}}
\newcommand{\plangle}{\moverlay{(\cr<}}
\newcommand{\prangle}{\moverlay{)\cr>}}
\newcommand{\tZ}{\ensuremath{\widetilde{Z}}}

\newcommand{\mat}[2]{\ensuremath{{#1}^{#2\times #2}}}
\newcommand{\mattuple}[3]{\ensuremath{\left({#1}^{#2\times #2}\right)^{#3}}}
\newcommand{\rmat}[3]{\ensuremath{{#1}}^{#2\times #3}}
\newcommand{\rmattuple}[4]{\ensuremath{\left({#1}^{#2\times #3}\right)^{#4}}}
\newcommand{\free}{\mathcal{F}}

\newcommand{\rs}{\mathcal{R}}
\newcommand{\ls}{\mathcal{L}}
\newcommand{\mtrans}{\boldsymbol{\top}}

\newcommand{\col}{\operatornamewithlimits{col}}
\newcommand{\row}{\operatornamewithlimits{row}}
\newcommand{\diag}{\operatornamewithlimits{diag}}

\theoremstyle{plain}
\newtheorem{thm}{Theorem}[section]
\newtheorem{cor}[thm]{Corollary}

\newtheorem{prop}[thm]{Proposition}
\theoremstyle{definition}
\newtheorem{dfn}[thm]{Definition}
\theoremstyle{remark}
\newtheorem{rem}[thm]{Remark}

\newtheorem{ex}[thm]{Example}

\newcommand{\spn}{\operatornamewithlimits{span}}
\newcommand{\ran}{\operatorname{ran}}
\numberwithin{equation}{section}
\newcommand{\emptyword}{\ensuremath{\boldsymbol{\emptyset}}}
\newcommand{\dom}{\operatorname{dom}}
\newcommand{\edom}{\operatorname{edom}}
\makeatletter
\def\moverlay{\mathpalette\mov@rlay}
\def\mov@rlay#1#2{\leavevmode\vtop{%
    \baselineskip\z@skip \lineskiplimit-\maxdimen
    \ialign{\hfil$#1##$\hfil\cr#2\crcr}}}

\begin{document}

\maketitle

\begin{abstract}
Noncommutative rational functions appeared in many contexts in
system theory and control, from the theory of finite automata and
formal languages to robust control and LMIs. We survey the
construction of noncommutative rational functions, their
realization theory and some of their applications. We also develop
a difference-differential calculus as a tool for further analysis.
\end{abstract}

\section{Introduction} \label{s:intro}

Noncommutative rational functions first appeared in system theory
in the context of  recognizable formal power series in
noncommuting indeterminates in the theory of formal languages and
finite automata; see Kleene \cite{Kle}, Sch\"{u}tzenberger
\cite{Schutz61,Schutz62b}, and Fliess
\cite{Fliess70,Fliess74a,Fliess74b} (where the motivation comes
also from applications to certain classes of nonlinear systems),
and Berstel--Reutenauer \cite{BR} for a survey. In particular,
noncommutative rational functions admit a good state space
realization theory. More recently, state space realizations of
rational expressions in Hilbert space operators (modelling
structured possibly time varying uncertainty) have figured
prominently in work on robust control of linear systems, see Beck
\cite{Beck}, Beck--Doyle--Glover \cite{BeckDoyleGlover},
Lu--Zhou--Doyle \cite{LuZhouDoyle}.

Another important application comes from the area of Linear Matrix
Inequalities (LMIs); see, e.g., Nesterov--Nemirovski \cite{NN},
Nemirovski \cite{N06}, Skelton--Iwasaki--Grigoriadis \cite{SIG97}.
As it turns out, most optimization problems appearing in systems
and control are dimension-independent, i.e., the natural variables
are matrices, and the problem involves rational expressions in
these matrix variables which have therefore the same form
independent of matrix sizes; see Helton \cite{H03},
Helton--McCullough--Putinar--Vinnikov \cite{HMcCPV}. Realizations
of rational functions in noncommuting indeterminates are exactly
what is needed here to convert (numerically unmanageable) rational
matrix inequalities into (highly manageable) linear matrix
inequalities, see Helton--McCullough--Vinnikov \cite{HMcCV}.

Last but not least, in many situations one can establish a
commutative result by ``lifting'' to the noncommutative setting,
applying the noncommutative theory, and then ``descending'' again
to the commutative situation. Some examples are:
\begin{itemize}
    \item The classical paper of
Fornasini--Marchesini \cite{FM} establishing a state space
realization theorem for rational functions of several commuting
variables.
    \item The results of Ball--Kaliuzhnyi-Verbovetskyi \cite{BK-V}
    on conservative dilations of various classes  of (commutative) multidimensional
    systems.
    \item The theorem of Kaliuzhnyi-Verbovetskyi--Vinnikov \cite{KVV} showing that
the singularities of a matrix-valued rational function of several
commuting variables which is regular at zero coincide with the
singularities of the resolvent in any of its Fornasini--Marchesini
realizations with the minimal possible state space dimension. This
implies, in particular, the absence of zero-pole cancellations in
a minimal factorization.

\end{itemize}

The goal of this paper is two-fold. First, we survey the basic
concepts of the theory of noncommutative rational functions, and
their realization theory. Second, we develop a
difference-differential calculus for noncommutative rational
functions. This is a new powerful tool for the needs of system
theory and beyond. It is a special instance of the general theory
of noncommutative functions which are defined as functions on
tuples of matrices of all sizes satisfying certain compatibility
conditions as we vary the size of matrices (they respect direct
sums and simultaneous similarities); see
Kaliuzhnyi-Verbovetskyi--Vinnikov \cite{ncfound}.

It is important to notice that the NCAlgebra software,
\begin{center}
http://www.math.ucsd.edu/$\sim$ncalg,
\end{center}
 implements many symbolic algorithms in the noncommutative
setting; see \cite{HMcCPV} for examples, guidance, and detailed
references.

\section{Noncommutative rational functions} \label{s:ncfun}

We first formally introduce noncommutative rational functions;
this involves some non-trivial details since
unlike the commutative case, a noncommutative rational function
does not admit a canonical coprime
fraction representation. We follow
Kaliuzhnyi-Verbovetskyi--Vinnikov \cite{KVV}, to which we
refer for both details and references to extensive algebraic
literature; we only mention Amitsur \cite{Am66},
Bergmann \cite{Be70}, Cohn \cite{Co71a,Co72} for some of the
original constructions, and
Rowen \cite[Chapter 8]{Row80}, Cohn \cite{Co71,Co06} for good
expositions.

We start with noncommutative polynomials in $d$ noncommuting
indeterminates $z_1,\ldots,z_d$ over a field  ${\field}$. E.g., a
noncommutative polynomial of total degree 2 in 2 indeterminates
$z_1,z_2$ is of the form
\begin{displaymath}
p=\alpha + \beta z_1 + \gamma z_2 + \delta z_1^2 + \epsilon z_1
z_2 + \zeta z_2 z_1 + \eta z_2^2,
\end{displaymath}
where the coefficients
$\alpha,\beta,\gamma,\delta,\epsilon,\zeta,\eta \in {\field}$. The
general form of a noncommutative polynomial is
$$p=\sum_{w\in\free_d}p_wz^w.$$
Here $\free_d$ denotes the free semigroup with $d$ generators
(letters) $g_1$, \ldots, $g_d$; elements of $\free_d$ are
arbitrary words $w=g_{i_\ell}\cdots g_{i_1}$ and the semigroup
operation is concatenation; the neutral element is the empty word
$\emptyword$, and $|w|=\ell$ is the length of the word $w$. We use
noncommutative multipowers $z^w=z_{i_\ell}\cdots z_{i_1}$. In the
example above, $$p_{\emptyword}=\alpha,\ p_{g_1}=\beta,\
p_{g_2}=\gamma,\ p_{g_1^2}=\delta,\ p_{g_1g_2}=\epsilon,\
p_{g_2g_1}=\zeta,\ p_{g_2^2}=\eta.$$

Noncommutative polynomials form an algebra ${\field}\langle
z_1,\ldots,z_d \rangle$ over ${\field}$, often called the free
associative algebra on $d$ generators $z_1,\ldots,z_d$. Notice
that we can evaluate a noncommutative polynomial $p \in
{\field}\langle z_1,\ldots,z_d \rangle$ on a $d$-tuple
$Z=(Z_1,\ldots,Z_d)$ of $n \times n$ matrices over ${\field}$, for
any $n$, yielding a $n \times n$ matrix $p(Z)$.

A non-zero polynomial can vanish on tuples of matrices of a
certain size. E.g., $p=z_1 z_2 - z_2 z_1$ vanishes on pairs of $1
\times 1$ matrices (scalars), and
$$p=\sum_{\pi\in
\mathcal{S}_{n+1}}\operatorname{sign}(\pi)\,x_1^{\pi(1)-1}x_2\cdots
x_1^{\pi(n+1)-1}x_2$$ vanishes on pairs of $n\times n$ matrices
(here  $\mathcal{S}_{n+1}$ is the symmetric group on $n+1$
elements); see \cite[Proposition 1.1.37 and Exercise 1.4.11 on
page 104]{Row80} and \cite[Theorem 7]{Form}. However, if $p(Z)=0$
for all $d$-tuples $Z$ of square matrices of all sizes, then
necessarily $p$ is the zero polynomial. More precisely, if
$p(Z)=0$ for all $d$-tuples $Z$ of $n \times n$ matrices, then
$\deg p \geq 2n$; this follows by applying to $p$ a
multilinearization process to reduce to the case of a polynomial
of degree $1$ in each indeterminate, and then evaluating the
resulting multilinear polynomial on a staircase of matrix units,
see \cite[page 6 and Lemma 1.4.3]{Row80}. We notice that a much
stronger statement appears in \cite{AK1}; in particular, the
intersection of the kernels of the matrix evaluations $p(Z)$ for
all $d$-tuples of $n\times n$ matrices is zero for $n$ large
enough compared to the degree of $p$.

We next define (scalar) noncommutative rational expressions by
starting with noncommutative polynomials and then applying
successive arithmetic operations --- addition, multiplication, and
inversion. We emphasize that an expression includes the order in
which it is composed and no two distinct expressions are
identified, e.g., $(z_1)+(-z_1)$, $(-1)+(((z_1)^{-1})(z_1))$, and
$0$ are different noncommutative rational expressions. A
noncommutative rational expression $r$ can be evaluated on a
$d$-tuple $Z$ of $n \times n$ matrices in its {\em domain of
regularity}, $\dom{r}$, which is defined as the set of all
$d$-tuples of square matrices of all sizes such that all the
inverses involved in the calculation of $r(Z)$ exist. E.g., if
$r=(z_1 z_2 - z_2 z_1)^{-1}$ then $\dom{r} = \{ Z=(Z_1,Z_2):\ \det
(Z_1Z_2-Z_2Z_1)\neq 0\}$. We assume that $\dom{r} \neq \emptyset$,
in other words, when forming noncommutative rational expressions
we never invert an expression that is nowhere invertible.

Two noncommutative rational expressions $r_1$ and $r_2$ are called
{\em equivalent} if $\dom{r_1} \cap \dom{r_2} \neq
\emptyset$\footnote{This requirement is in fact superfluous. For
any noncommutative rational expression $r$, it turns out
\cite[Remark 2.3]{KVV} that the evaluation of $r$ on $d$-tuples of
$n\times n$ generic matrices
--- see the discussion following Proposition \ref{prop:nc-invert}
below --- is defined for all sufficiently large $n$. In the case
where $\field$ is an infinite field it follows that $\dom_nr$ is
Zariski dense for all sufficiently large $n$; therefore for any
two noncommutative rational expressions $r_1$ and $r_2$,
$\dom{r_1} \cap \dom{r_2} \neq \emptyset$. The case where the
field $\field$ is finite can be handled as in the proof of
Proposition \ref{prop:nc-invert}.} \label{page:reason} and $r_1(Z)
= r_2(Z)$ for all $d$-tuples $Z \in \dom{r_1} \cap \dom{r_2}$.
E.g., the three different noncommuting rational expressions in the
paragraph above are equivalent. For another example, easy matrix
algebra shows that $r_1=z_1 z_2 (z_1z_2 - z_2 z_1)^{-1}$ and
$r_2=1+z_2z_1(z_1z_2 - z_2z_1)^{-1}$ are equivalent.

We define a {\em noncommutative rational function} to be an
equivalence class of noncommutative rational expressions. We
usually denote noncommutative rational functions by German
(Fraktur) letters.

Notice that, unlike in the commutative case, the ``minimal
complexity'' of a noncommutative rational expression defining a
given noncommutative rational function can be arbitrarily high;
there is nothing similar to a coprime fraction representation.

It turns out that any nonzero noncommutative rational function is
invertible. This follows from the following result which is
essentially well known and is non-trivial already in the case
where $r$ in the statement is a noncommutative polynomial.
\begin{prop}\label{prop:nc-invert}
If $r$ is a noncommutative rational expression and $\det r(Z)=0$
for all $Z\in\dom{r}$ then $r$ is equivalent to zero.
\end{prop}

Proposition \ref{prop:nc-invert} means that noncommutative
rational functions form a skew field
--- a skew field of fractions of the ring of noncommutative
polynomials. We remind the reader that a skew field, also called a
division ring, is a ring with identity in which every nonzero
element has a multiplicative inverse; it is therefore similar to a
field, except that the multiplication is not assumed to be
commutative. A skew field of fractions of a given ring is a skew
field containing the ring and generated by it in the sense that no
proper skew subfield contains the ring. If a noncommutative
integral domain $\ring$ satisfies the so called right Ore
condition,
\begin{equation}\label{eq:Ore}
\forall\, a, b\in\ring, b \neq 0\ \exists\, c, d\in\ring, d \neq
0\colon a d = b c,
\end{equation}
then one can construct a skew field of fractions analogously to
the commutative case as the ring of right quotients, i.e., of
formal fractions  $c d^{-1}$, $d \neq 0$.  In general, a skew
field of fractions of a noncommutative integral domain might or
might not exist. We refer to \cite[Sections 0.5--0.6 and Chapter
7]{Co71}, \cite{Co06}, \cite[Chapter 4]{Lamb}, \cite[Section 1.7
and pages 105, 107--108]{Row80} for more material on skew fields
of fractions and noncommutative localization. If a skew field of
fractions exists, it might not be unique; we will discuss this in
more detail later in this section.

Before proceeding to the proof of Proposition
\ref{prop:nc-invert}, we notice that for  a given matrix size $n$,
we can view a noncommutative polynomial or a noncommutative
rational expression in $d$ noncommuting indeterminates as a
$n\times n$ matrix-valued polynomial or rational function in
commuting matrix entries. More formally, let $T=(T_1,\ldots,T_d)$
be a $d$-tuple of $n\times n$ matrices whose entries $(T_i)_{jk}$
are $dn^2$ commuting indeterminates; $T_1$, \ldots, $T_d$ are
often called generic matrices. For a noncommutative polynomial $p$
or for a noncommutative rational expression $r$, we define
$$p_n=p(T)\in\mat{\field}{n}\big[(T_i)_{jk} \colon {i=1,\ldots,d;\
j,k=1,\ldots,n}\big]$$ and
$$r_n=r(T)\in\mat{\field}{n}\big((T_i)_{jk} \colon {i=1,\ldots,d;\
j,k=1,\ldots,n}\big).$$ Note that $r_n$ is defined only for $n$ in
a subset $\mathcal{N}_r\subseteq\mathbb{N}$ such that all the
inversions involved in the calculation of the rational
matrix-valued function $r(T)$ exist. E.g., if $r=(z_1 z_2 - z_2
z_1)^{-1}$ then $\mathcal{N}_r=\mathbb{N}\setminus\{ 1\}$.

We make two useful remarks. First, if $\field$ is an infinite
field then $n\in\mathcal{N}_r$ if and only if $\dom_nr:=\dom
r\cap\mattuple{\field}{n}{d}$ is nonempty. This may fail in the
case of a finite field $\field$. E.g., if
$\field=\mathbb{Z}/p\mathbb{Z}$ and $r(z_1)=(z_1^p-z_1)^{-1}$ then
$1\in\mathcal{N}_r$, however $\dom_1r=\emptyset$. The problem is,
of course, in that there are nonzero polynomials over $\field$
which vanish identically.

Second, we define the \emph{extended domain of regularity}, $\edom
r$, of a noncommutative rational expression $r$ as follows: for
each matrix size $n$, $\edom_nr:=\edom
r\cap\mattuple{\field}{n}{d}$ is the domain of regularity of the
rational matrix-valued function $r_n$. (We set
$\edom_nr=\emptyset$ if $n\notin\mathcal{N}_r$.) Here the domain
of regularity of a rational matrix-valued function of commuting
indeterminates is defined to be the intersection of the domains of
regularity of its entries; the domain of regularity of a scalar
rational function is the complement of the zero set of its
denominator in the coprime fraction representation. In general,
$\dom r \subsetneq\edom r$. As a silly example, take
$r=z_1z_1^{-1}$. Then $\dom r=\{ Z_1\colon \det Z_1\neq 0\}$,
however $\mathcal{N}_r=\mathbb{N}$, $r_n=I_n$ for each $n$, so
that $\edom r$ consists of all square matrices $Z_1$ over $\field$
of all sizes.  For a more conceptual example, see the end of this
section.

\begin{proof}[Proof of Proposition \ref{prop:nc-invert}]
 We
consider a subring $G_n=\{ p_n\colon p\in\field\langle
z_1,\ldots,z_d\rangle\}$  of $\mat{\field}{n}[(T_i)_{jk}]$, which
is often called the ring of generic matrices. Let $D_n$ be the
ring of central quotients of $G_n$, i.e., the ring of formal
fractions $PQ^{-1}$ with $P,Q\in G_n$ and $Q$ a regular central
element (see \cite[Section 1.7]{Row80} for details). By a theorem
of Amitsur \cite[Theorem 3.2.6]{Row80}, $D_n$ is a skew field.
Assuming $d>1$,  any central element $Q\in G_n$ is central in the
whole ring $\mat{\field}{n}[(T_i)_{jk}]$ (see \cite[Exercise 2.4.2
on page 149]{Row80}). Hence, any such nonzero $Q$ is a scalar
$n\times n$ matrix of polynomials in $(T_i)_{jk}$; in particular,
$Q$ is invertible in $\mat{\field}{n}\big((T_i)_{jk}\big)$. It
follows that the skew field $D_n$ is a subring of
$\mat{\field}{n}\big((T_i)_{jk}\big)$ (clearly, this is also true
in the case of $d=1$, where $G_n$ is commutative, every element is
central, and every element is invertible in
$\mat{\field}{n}\big((T_1)_{jk}\big)$). Therefore, $r_n\in D_n$
for any noncommutative rational expression $r$ with
$n\in\mathcal{N}_r$. If  $\det r_n=0$ then $r_n$ is not invertible
in $\mat{\field}{n}\big((T_i)_{jk}\big)$. On the other hand, since
$D_n$ is a skew field, this can happen only if $r_n=0$.

Assume now that $\det r(Z)=0$ for all $Z\in\dom{r}$. We claim that
$\det r_n=0$ for each matrix size $n\in\mathcal{N}_r$; by the
preceding paragraph, this will imply the conclusion of the
proposition: $r(Z)=0$ for all $Z\in\dom_n r$, $n\in\mathcal{N}_r$,
i.e., $r$ is equivalent to zero. If the field $\field$ is infinite
then the assumption $\det r(Z)=0$ for all $Z\in\dom_n r$ is simply
equivalent to $\det r_n=0$ (as a rational function of
$(T_i)_{jk}$).

Let now $\field$ be a finite field, and assume that $\det r_n\neq
0$. Then there exists a finite extension ${\tfield}$ of $\field$,
say of degree $m$, and $\tZ\in \mattuple{\tfield}{n}{d}$ such that
$\tZ\in\dom r$ (over $\tfield$) and $\det r(\tZ)\neq 0$. Since
$\field$ is finite, it is perfect, therefore every finite
extension is separable, and thus admits a primitive element
$\alpha$, i.e., $\tfield=\field(\alpha)$ (see \cite[Theorem 7.6.1
and Corollary 7.7.8]{Lang}). Then $1$, $\alpha$, \ldots,
$\alpha^{m-1}$ is a basis for $\tfield$ over $\field$, and we can
define
 a ring homomorphism
$\phi \colon {\tfield} \to \mat{\field}{m}$ by setting
$\phi(\alpha)=A$ where  $A$ is a $m \times m$ matrix over $\field$
whose minimal polynomial coincides with that of $\alpha$. $\phi$
induces a ring homomorphism $\phi_n = {\rm id}_{\mat{\field}{n}}
\otimes_{\field} \phi$ from
$\mat{\tfield}{n}\cong\mat{\field}{n}\otimes_{\field} \tfield$ to
$\mat{\field}{n}\otimes_{\field}\mat{\field}{m}\cong\mat{\field}{nm}$.
The fact that $\phi_n$ is a homomorphism implies that
$(\phi_n(\tZ_1),\ldots,\phi_n(\tZ_d))\in\dom r$ and
$$r(\phi_n(\tZ_1),\ldots,\phi_n(\tZ_d)) =
\phi_n(r(\tZ_1,\ldots,\tZ_d)).$$ For any
$\widetilde{X}=[\widetilde{x}_{ij}]\in\mat{\tfield}{n}$, we claim
that
$$\det\phi(\det\widetilde{X})=\det\phi_n(\widetilde{X}),$$
which boils down to
$$\det\left(\sum_{\pi\in
S_n}\operatorname{sign}(\pi)\phi(\widetilde{x}_{1\pi(1)})\cdots
\phi(\widetilde{x}_{n\pi(n)})\right)=\det\left[\phi(\widetilde{x}_{ij})\right]_{i,j=1,\ldots,n}.$$
Since the matrices $\phi(\widetilde{x}_{ij})$ commute, the last
formula follows from a well known identity for the determinant of
a block matrix with commuting blocks; see, e.g., \cite{KSW}.
Finally,
\begin{multline*}
\det r\left(\phi_n(\tZ_1),\ldots,\phi_n(\tZ_d)\right)=\det
\phi_n\left(r(\tZ_1,\ldots,\tZ_d)\right)\\ =\det\phi\left(\det
r(\tZ_1,\ldots,\tZ_d)\right).
\end{multline*}
Since $\det r(\tZ)\neq 0$, and $\phi$ is a ring homomorphism whose
domain is a field, $\phi\left(\det r(\tZ)\right)$ is invertible.
Therefore, $$\det
r\left(\phi_n(\tZ_1),\ldots,\phi_n(\tZ_d)\right)=\det\phi\left(\det
r(\tZ)\right)\neq 0,$$ which contradicts the assumption that $\det
r(Z)=0$ for each $Z\in\dom r$.
 \end{proof}

The proof of Proposition \ref{prop:nc-invert} implies two
interesting facts. First, while the ``minimal complexity'' of a
noncommutative rational expression defining a given noncommutative
rational function can be arbitrarily high, its restriction to
$n\times n$ matrices, for every matrix size $n$, is of a simple
form: it is equal to the restriction of $pq^{-1}$, where $p$ and
$q$ are noncommutative polynomials, $q$ being a central polynomial
for $n\times n$ matrices.

Second, we could have defined noncommutative rational expressions
and their equivalence using evaluation on generic matrices or on
matrices over the algebraic closure of $\field$ rather than
evaluation on matrices over $\field$ as we did. However, we would
obtain the same noncommutative rational expressions and the same
equivalence relation.

We denote the skew field of noncommutative rational functions in
$z_1$, \ldots, $z_d$ by $\field\plangle z_1,\ldots,z_d\prangle$;
it is often called the free skew field.

Unlike in the commutative case, skew fields of fractions are in
general not unique. Here is an example of infinitely many
embeddings of $\field\langle z_1,z_2\rangle$ into skew fields (see
\cite{Fish} and \cite[Exercise 7.2.10 on page 258]{Co71}).
Consider the polynomial ring in one indeterminate $\field [t]$
with the endomorphism $\alpha_n$ ($n=2,3,\ldots$) induced by
$\alpha_n(t)=t^n$ and let $\field [t][x;\alpha_n]$ be the skew
polynomial ring over $\field [t]$ determined by $\alpha_n$. The
elements of $\field [t][x;\alpha_n]$ are polynomials over $\field$
in $t$ and $x$ with the indeterminates $t$ and $x$ satisfying the
commutation relation $tx=xt^n$. Like any skew polynomial ring,
$\field [t][x;\alpha_n]$ is a right Ore ring, i.e., satisfies
\eqref{eq:Ore}, and can be embedded into its skew field of right
quotients, $\field (t)(x;\alpha_n)$. The elements of $\field
(t)(x;\alpha_n)$ are rational functions over $\field$ in $t$ and
$x$ with $tx=xt^n$. See \cite[Section 0.8]{Co71} for details on
skew polynomial rings. Let $y=xt$. It turns out that for any
noncommutative polynomial $p$ in two indeterminates, $p(x,y)\neq
0$. This can be verified directly by showing that distinct
noncommutative monomials in $x$ and $y$ yield distinct monomials
of the form $x^kt^\ell$; it can also be deduced from the result of
\cite{Ja}, since the two left ideals in $\field [t][x;\alpha_n]$
generated by $x$ and by $y$ have a trivial intersection. We thus
obtain an embedding $p\mapsto p(x,y)$ of $\field \langle
z_1,z_2\rangle$ into $\field [t][x;\alpha_n]$ and, therefore, into
$\field (t)(x;\alpha_n)$. It is not hard to see that these
embeddings are not isomorphic for distinct values of $n$.

It is instructive to observe that in the embedding $\field\langle
z_1,z_2\rangle\hookrightarrow\field (t)(x;\alpha_n)$ the images
$x$ and $y$ of $z_1$ and $z_2$ satisfy the rational identity
$x^{-1}yx=y(x^{-1}y)^{n-1}$ (since $tx=xt^n$ and $y=xt$, hence
$t=x^{-1}y$). This is in contrast to the free skew field
$\field\plangle z_1,z_2\prangle$, where $z_1$, $z_2$ satisfy no
nontrivial rational identities.

The non-uniqueness issue leads us to the notion of the universal
skew field of fractions. A skew field of fractions $\skfield$ of a
ring $\ring$ is called a {\em universal skew field of fractions}
if for every homomorphism $\phi\colon\ring\to\mathbf{L}$ to a skew
field $\mathbf{L}$ there exists a subring
$\skfield_0\subseteq\skfield$ containing $\ring$ and a
homomorphism $\theta\colon\skfield_0\to\mathbf{L}$ extending
$\phi$ such that the following holds:
\begin{equation}\label{def:ext}
\text{\em for every $x\neq 0$ in $\skfield_0$ its inverse $x^{-1}$
belongs to $\skfield_0$ if and only if $\theta(x)\neq 0$.}
\end{equation}
Furthermore, the extension $\theta$ is unique in the following
local sense. Let $\theta'\colon\skfield_0'\to\mathbf{L}$ be
another extension satisfying \eqref{def:ext}. Then there exists a
subring $\widetilde{\skfield}_0\subseteq\skfield_0\cap\skfield_0'$
containing $\ring$ such that $\theta$ and $\theta'$ agree on
$\widetilde{\skfield}_0$, and the extension
$\widetilde{\theta}\colon\widetilde{K}_0\to\mathbf{L}$ defined by
$\widetilde{\theta}:=\theta|_{\widetilde{\skfield}_0}=\theta'|_{\widetilde{\skfield}_0}$
satisfies \eqref{def:ext}. It is straightforward to see that a
universal skew field of fractions is unique (when it exists) up to
a unique isomorphism.

The following result is essentially the first fundamental theorem
of Amitsur on rational identities.
\begin{prop}\label{prop:univer}
$\field\plangle z_1,\ldots,z_d\prangle$ is the universal skew
field of fractions of the ring $\field\langle
z_1,\ldots,z_d\rangle$.
\end{prop}
\begin{proof}
 Let $\phi \colon {\field}\langle
z_1,\ldots,z_d \rangle \to {\mathbf L}$ be a homomorphism to a
skew field $\mathbf{L}$. We notice that for any skew field
$\mathbf{D}$ over $\field$, a rational expression $r$ can be
evaluated on a $d$-tuple $x=(x_1,\ldots,x_d)$ of elements of
$\mathbf{D}$ provided all the inverses involved in the calculation
of $r(x)$ exist, i.e., all the elements to be inverted are
nonzero. Therefore, we can define $\theta(r)$ as the evaluation
$r(\phi(z_1),\ldots,\phi(z_d))$ whenever this is possible.

We claim that if $r$ is equivalent to $0$ then
$r(\phi(z_1),\ldots,\phi(z_d))$ is either $0$ or undefined.
Indeed, $r$ being equivalent to $0$ means that for every matrix
size $n$, $r$ is either $0$ or undefined on the skew field of
fractions $D_n$ of the ring of generic matrices $G_n$, see the
proof of Proposition \ref{prop:nc-invert}. In other words, $r$
 is a rational identity for $D_n$, $n=1,2,\ldots$, hence (see
\cite[Theorem 8.3.3 and Corollary 8.2.16]{Row80})  $r$ is a
rational identity for any skew field over ${\field}$. In a little
bit more details, the fact that $r$
 is a rational identity for $D_n$, $n=1,2,\ldots$ (or just for a sequence $D_{n_j}$, $n_j\to\infty$),
implies by a simple ultraproduct construction in the proof of
Corollary 8.2.16 in \cite{Row80} that $r$ is a rational identity
for a skew field ${\mathbf D}$  that is infinite dimensional over
an infinite center; hence by the first fundamental theorem of
Amitsur \cite[Theorem 8.2.15]{Row80}, $r$ is a rational identity
for any skew field over ${\field}$. At any rate, $r$ is a rational
identity for ${\mathbf L}$, hence, $r(\phi(z_1),\ldots,\phi(z_d))$
is either $0$ or undefined.

We can now define $\skfield_0$ to consist of all noncommutative
rational functions $\mathfrak{r}$ that can be represented by
noncommutative rational expressions $r$ such that $\theta(r)$ is
defined, and we set $\theta(\mathfrak{r}):=\theta(r)$. It is clear
that $\skfield_0$ is a subring of $\field\plangle
z_1,\ldots,z_d\prangle$ containing $\field\langle
z_1,\ldots,z_d\rangle$, that $\theta\colon\skfield_0\to\mathbf{L}$
is a homomorphism extending $\phi$ and satisfying \eqref{def:ext}.
Furthermore, $\theta$ is the only extension of $\phi$ to
$\skfield_0$. Hence, if $\theta'\colon\skfield_0'\to\mathbf{L}$ is
another extension of $\phi$ satisfying \eqref{def:ext}, then
$\theta$ coincides with $\theta'$ on
$\widetilde{\skfield}_0=\skfield_0\cap\skfield_0'$, and it is
obvious that the extension
$\widetilde{\theta}\colon\widetilde{\skfield}_0\to\mathbf{L}$
defined by
$\widetilde{\theta}:=\theta|_{\widetilde{\skfield}_0}=\theta'|_{\widetilde{\skfield}_0}$
satisfies \eqref{def:ext}.
\end{proof}

Finally, we introduce matrix-valued noncommutative rational
expressions and matrix-valued noncommutative rational functions.
We start with matrix-valued noncommutative polynomials (having
matrix rather than scalar coefficients) and use tensor
substitutions for evaluations on tuples of matrices. E.g., if
\begin{equation*}
P = P_{\emptyword} + P_{g_1} z_1 + P_{g_2}z_2 +
P_{g_1^2}z_1^2+P_{g_1g_2}z_1z_2+P_{g_2g_1}z_2z_1+P_{g_2^2}z_2^2
\end{equation*}
is a matrix-valued noncommutative polynomial of total degree $2$
with coefficients $P_w\in \rmat{\field}{p}{q}$, $|w|\le 2$, then
for a $d$-tuple $Z=(Z_1,\ldots,Z_d)$ of $n \times n$ matrices over
$\field$,
\begin{multline*}
\hspace{-.35cm}P(Z) = P_{\emptyword}\otimes I_n + P_{g_1}\otimes
Z_1 + P_{g_2}\otimes Z_2 + P_{g_1^2}\otimes
Z_1^2+P_{g_1g_2}\otimes
Z_1Z_2+P_{g_2g_1}\otimes Z_2Z_1+P_{g_2^2}\otimes Z_2^2\\
\in\rmat{\field}{p}{q}\otimes\mat{\field}{n}.
\end{multline*}
We will often use the canonical identification of
$\rmat{\field}{p}{q}\otimes\mat{\field}{n}$ with
$\rmat{\field}{pn}{qn}$, i.e., with $p\times q$ block matrices
with $n\times n$ block entries. \label{page:canon} Thus we will
often view $P(Z)$ above as a $pn \times qn$ matrix. We define
matrix-valued noncommutative rational expressions by starting with
matrix-valued noncommutative polynomials and applying successive
matrix arithmetic operations
--- addition, multiplication, and inversion, and forming block
matrices: a $p_1 \times q_1$ matrix of $p_2 \times q_2$
matrix-valued noncommutative rational expressions is a $p_1 p_2
\times q_1 q_2$ matrix-valued noncommutative rational expression.
 The domain of a matrix-valued noncommutative
rational expression $R$, $\dom R$, consists of all $d$-tuples $Z$
of square matrices of all sizes such that all the inverses
involved in the calculation of $R(Z)$ exist. E.g., consider the $1
\times 1$ matrix-valued rational expression \label{three_rats}
$$
R_1=\begin{bmatrix} 1 & 0 \end{bmatrix}
\begin{bmatrix} 1-z_1 & -z_2 \\ -z_2 & 1-z_1\end{bmatrix}^{-1}
\begin{bmatrix} 1 \\ 0 \end{bmatrix},
$$
with
$$
\dom{R_1}=\left\{ (Z_1,Z_2): \det \begin{bmatrix} I-Z_1 & -Z_2 \\
-Z_2 & I-Z_1\end{bmatrix} \neq 0\right\}.
$$
Notice that a $1\times 1$ matrix-valued noncommutative rational
expression is not necessarily the same as a scalar noncommutative
rational expression, since it may involve, as in this example,
intermediate matrix operations.

Equivalence of matrix-valued noncommutative rational expressions,
and matrix-valued noncommutative rational functions as equivalence
classes, are defined as in the scalar case. E.g., using a standard
Schur complement calculation, we can observe that the $1\times 1$
matrix-valued noncommutative rational expression $R_1$ above is
equivalent to the following two scalar noncommutative rational
expressions,
$$
r_2=(1-z_1-z_2 (1-z_1)^{-1}z_2)^{-1}
$$
and
$$
r_3=-z_2^{-1}(1-z_1)(z_2-(1-z_1)z_2^{-1}(1-z_1))^{-1},
$$
with
\begin{equation*}
\dom{r_2}=\{ (Z_1,Z_2): \det (I-Z_1)\neq 0,\ \det (I-Z_1-Z_2
(I-Z_1)^{-1}Z_2)\neq 0\}
\end{equation*}
and
\begin{equation*}
\dom{r_3}=\{ (Z_1,Z_2):\ \det (Z_2)\neq 0,\ \det
(Z_2-(I-Z_1)Z_2^{-1}(I-Z_1))\neq 0\}.
\end{equation*}

It is not \emph{a priori} clear whether a $p\times q$
matrix-valued noncommutative rational function is the same thing
as a $p\times q$ matrix of (scalar) noncommutative rational
functions; the question is whether any $p\times q$ matrix-valued
noncommutative rational function can be represented by a $p\times
q$ matrix of scalar noncommutative rational expressions. It turns
out that this is true, because noncommutative rational functions
form a skew field; see \cite[Remarks 2.16 and 2.11]{KVV} for
details.

We define the {\em domain of regularity of a matrix-valued
noncommutative rational function ${\mathfrak R}$} as the union of
the domains of regularity of all matrix-valued noncommutative
rational expressions representing this function, i.e.,
\begin{displaymath}
\dom{{\mathfrak R}} = \bigcup_{R \in {\mathfrak R}} \dom R.
\end{displaymath}
We emphasize that even for the case of a (scalar) noncommutative
rational function, we define its domain using all $1\times 1$
matrix-valued noncommutative rational expressions representing the
function, not just the scalar ones. E.g., in the examples above,
it is easily seen that $\dom r_2$ and $\dom r_3$ are both properly
contained in $\dom R_1$; so, if $\mathfrak{r}$ is the
corresponding noncommutative rational function, then $\dom
\mathfrak{r}\supseteq\dom R_1$. In fact, the result on the
singularities of minimal realization (to be discussed in Section
\ref{s:real}) implies that $\dom \mathfrak{r}=\dom R_1$. See
\cite[Remark 2.11]{KVV} for additional discussion and references.

We can also evaluate a matrix-valued noncommutative rational
expression $R$ on generic matrices as in the discussion preceding
the proof of Proposition~\ref{prop:nc-invert}, and introduce a
subset $\mathcal{N}_R\subseteq\mathbb{N}$ where the evaluation is
defined and the extended domain $\edom R$.  We then define
$\mathcal{N}_{\mathfrak{R}}$ and $\edom{{\mathfrak R}}$, the {\em
extended domain of regularity of a matrix-valued noncommutative
rational function ${\mathfrak R}$}, by
\begin{displaymath}
\mathcal{N}_\mathfrak{R}=\bigcup_{R \in {\mathfrak R}}
\mathcal{N}_R,\quad
 \edom{{\mathfrak R}} =
\bigcup_{R \in {\mathfrak R}} \edom R.
\end{displaymath}
We notice that while in general, for $R \in {\mathfrak R}$, $\dom
R\subsetneq\dom\mathfrak{R}$,
 it is always the case that $\edom
R=\edom\mathfrak{R}$ provided that
$\mathcal{N}_R=\mathcal{N}_\mathfrak{R}$; see \cite[Section
2]{KVV} for additional discussion.

\section{Realization theory for noncommutative rational functions}
\label{s:real}

It is a bitter experience that the constellation of foundational
facts underlying the classical Kalman realization theory for 1D
systems collapses for rational functions of several commuting
variables and commutative multidimensional systems. It is all the
more amazing that these facts do hold, with obvious modifications,
in the noncommutative setting. Noncommutative systems of the form
\eqref{nc-FM} below were first studied by Ball--Vinnikov
\cite{Cuntz2} in the conservative setting, in the context of
operator model theory for row contractions (Popescu
\cite{Popescu-model1,Popescu-model2,Popescu-CLT1,Popescu-CLT2})
and of representation theory of the Cuntz algebra
(Bratelli--Jorgensen \cite{BJ} and Davidson--Pitts \cite{DP}).
 On the other hand,
noncommutative realizations very similar to \eqref{tf-FM} were
considered much earlier in the theory of formal languages and
finite automata in the work of Kleene, Sch\" utzenberger and
Fliess \cite{Kle,Schutz61,Fliess74a}. A comprehensive study of
noncommutative realization theory appears in
Ball--Groenewald--Malakorn \cite{BGM1,BGM2,BGM3}; these papers
give a unified framework of \emph{structured noncommutative
multidimensional linear systems} for different kinds of
realization formulae. We also mention the paper by
Ball--Kaliuzhnyi-Verbovetskyi \cite{BK-V} where an even more
general class of noncommutative systems (given though in a
frequency domain) was described and the corresponding dilation
theory was developed.

A noncommutative multidimensional system is a system with
evolution along the free semigroup $\free_d$ on $d$ letters
$g_1,\ldots,g_d$ rather than along the multidimensional integer
lattice ${\mathbb Z}^d$. An example of system equations with
evolution along $\free_d$ is given by a \emph{noncommutative
Fornasini--Marchesini system} (see \cite{FM} for the original
commutative version):
\begin{equation}  \label{nc-FM}
\Sigma^{\rm FM}\colon \left\{
\begin{array}{rcl}
x(g_{1}w) & = &  A_{1} x(w) + B_{1} u(w), \\
     & \vdots &  \\
     x(g_{d}w) & = & A_{d} x(w) + B_{d}u(w), \\
y(w) & = & C x(w) + D u(w),
\end{array}  \right. \qquad (w\in\free_d).
\end{equation}

Applying to the system equations \eqref{nc-FM} an appropriately
defined formal noncommutative $z$-transform and under the
assumption that the state of the system is initialized at $0$ (so
that $x(\emptyword) = 0$), we arrive at the input-output relation
$$
\widehat{y}(z) = T_{\Sigma^{\rm FM}}(z) \widehat{u}(z)
$$
where the {\em transfer function} is given by
\begin{equation} \label{tf-FM}
 T_{\Sigma^{\rm FM}}(z) = D + C(I_m-A_1z_1-\cdots-A_dz_d)^{-1} (B_1z_1+\cdots+B_dz_d).
\end{equation}
Here $m$ is the dimension of the state space $\field^m$ where
vectors $x(w)$ live. We see that the transfer function is a
matrix-valued noncommutative rational function in noncommuting
indeterminates $z_{1}, \dots, z_{d}$ which is regular at zero,
i.e., zero belongs to its domain of regularity (a little more
precisely, the transfer function is the matrix-valued
noncommutative rational function defined by the matrix-valued
noncommutative rational expression \eqref{tf-FM}).

The system \eqref{nc-FM} is called \emph{controllable} (resp.,
\emph{observable}) if
$$\spn_{w\in\free_d,\,j=1,\ldots,d}\ran\{A^wB_j\}=\field^m,\quad
{\rm (resp.,}\ \bigcap_{w\in\free_d}\ker \{CA^w\}=\{0\}).$$

The following facts are fundamental for the noncommutative
realization theory:

\begin{enumerate}

\item Every matrix-valued noncommutative rational function which
is regular at zero admits a state space realization \eqref{tf-FM}.

\item An arbitrary realization \eqref{tf-FM} of a given
matrix-valued noncommutative rational function can be reduced via
an analogue of the Kalman decomposition to a controllable and
observable realization.

\item A realization \eqref{tf-FM} is controllable and observable
if and only if it is minimal, i.e., it has the smallest possible
state space dimension, and a minimal realization is unique up to a
unique similarity.

\item A minimal realization \eqref{tf-FM} can be constructed
canonically and explicitly from a matrix-valued noncommutative
rational function by means of the corresponding Hankel operator;
this ties in with the fact that the Hankel operator corresponding
to a matrix-valued noncommutative formal power series has finite
rank if and only if the power series represents a rational
function (an analogue of Kronecker's Theorem).

\item In a minimal realization \eqref{tf-FM}, the singularities of
the transfer function coincide with the singularities of the
resolvent; more precisely, the domain of regularity\footnote{In
fact, this is also the extended domain of regularity of the
transfer function.} of the transfer function \eqref{tf-FM} is
exactly
\begin{displaymath}
\{(Z_1,\ldots,Z_d) \colon \det(I-A_1 \otimes Z_1-\cdots-A_d
\otimes Z_d) \neq 0\}.
\end{displaymath}
\end{enumerate}

 For the proofs of items (1)--(4), including missing details and exact
 references to the earlier literature,
  we refer to \cite{BGM1} where
these facts are established in a more general setting of
structured noncommutative multidimensional systems.

As for item (5), it is amazingly difficult to prove ``by hands'';
the usual proofs for $d=1$ use the Hautus test for
controllability~/ observability, but this is no longer available.
A proof appears in \cite{KVV} using noncommutative backward shifts
which are a particular instance of the difference-differential
calculus for noncommutative rational functions. This is a special
case of the difference-differential calculus for general
noncommutative functions, which are functions on tuples of square
matrices of all sizes which respect direct sums and simultaneous
similarities. The forthcoming basic reference is \cite{ncfound}.
The difference-differential calculus for noncommutative rational
functions can be developed in a more straightforward manner than
in the general case, and we will do this later in Section
\ref{s:difdif}.

Another important example of a structured noncommutative
multidimensional system is a \emph{noncommutative Givone--Roesser
system} (for the original commutative version of these systems,
see \cite{GR}):
\begin{equation}  \label{nc-GR}
\Sigma^{\rm GR}\colon \left\{
\begin{array}{rcl}
x_1(g_{1}w) & = &  A_{11} x_1(w)+\cdots + A_{1d}x_d(w)+B_{1} u(w), \\
     & \vdots &  \\
     x_d(g_{d}w) & = & A_{d1} x_1(w)+\cdots +A_{dd}x_d + B_{d}u(w), \\
y(w) & = & C_1 x_1(w)+\cdots +C_dx_d(w) + D u(w),
\end{array}  \right. \quad (w\in\free_d).
\end{equation}
Here $x_j\in\field^{m_j}$, i.e., the state space has $d$
components: $\field^m=\field^{m_1}\oplus\cdots\oplus\field^{m_d}$.
 The transfer function of the noncommutative
Givone--Roesser system is given by
\begin{equation} \label{tf-GR}
 T_{\Sigma^{\rm GR}}(z) = D + C(I_m-\Delta(z)A)^{-1} \Delta(z)B,
\end{equation}
where $A$ is a $d\times d$ block matrix with blocks
$A_{ij}\in\rmat{\field}{m_i}{m_j}$, and $\Delta(z)$ is a $d\times
d$ block diagonal matrix, with matrix-valued noncommutative
monomials $I_{m_1}z_1$, \ldots, $I_{m_d}z_d$ on the diagonal.

The system \eqref{nc-GR} is called \emph{controllable} (resp.,
\emph{observable}) if
$$\spn_{w\in\free_d}\ran\{P_jA^wB\}=\field^{m_j},\quad
{\rm (resp.,}\ \bigcap_{w\in\free_d}\ker
\{CA^w|_{\field^{m_j}}\}=\{0\}),\quad j=1,\ldots,d.$$ Here $P_j$
is the orthogonal projection of the state space $\field^m$ onto
its $j$-th component $\field^{m_j}$.

As we already mentioned, items (1)--(4) above hold for arbitrary
structured noncommutative multidimensional system realizations, in
particular for the noncommutative Givone--Roesser realization
\eqref{nc-GR}. The result in \cite{KVV}, i.e., item (5), has been
proved for a much more general class of realizations than
\eqref{tf-FM}; however, this class does not cover all structured
noncommutative multidimensional system realizations, and we
conjecture that the result might fail for noncommutative
Givone--Roesser realizations.

On the other hand, the symmetry appearing in Givone--Roesser
system equations\footnote{In the case of $\field=\mathbb{C}$, the
adjoint system $\Sigma^{*GR}$ has the same form as $\Sigma^{GR}$,
but with switching the input and output spaces and
replacing the coefficient block matrix $\begin{bmatrix} [A_{ij}] & \col[B_j]\\
\row[C_i] & D
\end{bmatrix}$  by its adjoint.} makes Givone--Roesser realizations more suitable
 for problems where this
symmetry is essential. For general structured noncommutative
multidimensional systems, basic arithmetic operations on transfer
functions (sum, product, inversion) correspond to certain
operations on systems, in the same manner as it occurs in the
classical 1D case. For noncommutative Givone--Roesser systems, we
have that, in addition, the adjoint of the transfer function is
the transfer function of the adjoint system. Exploiting these
correspondences, one can study noncommutative rational functions
with certain symmetries in terms of their realizations.

In the paper by Alpay--Kaliuzhnyi-Verbovetskyi \cite{AK2}, classes
of matrix-valued noncommutative rational functions with various
symmetries were studied in terms of their Givone--Roesser
realizations. A sample result from \cite{AK2} is a version of the
so-called lossless bounded real lemma (cf. \cite{BGM3} for the
general bounded real lemma in the noncommutative setting). Let $F$
be a $q\times q$ matrix-valued noncommutative rational function
over the field $\mathbb{C}$ which is regular at zero. Let
$J=J^{-1}=J^*\in\mat{\mathbb{C}}{q}$. Then $F$ is called
\emph{matrix-$J$-unitary} on the set $\mathcal{J}_d$ of $d$-tuples
of skew-Hermitian $n\times n$ matrices, $n=1,2,\ldots$ (which is a
noncommutative analogue of the imaginary axis of the complex
plain) if
\begin{equation}\label{eq:j-unitary}
 F(Z)(J\otimes I_n)F(Z)^*=J\otimes I_n \qquad (Z\in \mathcal{J}_d)
\end{equation}
at all points $Z\in \mathcal{J}_d\cap\dom F$. Suppose that $F$ is
a $q\times q$ matrix-valued noncommutative rational function over
$\mathbb{C}$ which is regular at zero,  and let
\eqref{nc-GR}--\eqref{tf-GR} be its minimal noncommutative
Givone--Roesser system realization. Then $F$ is matrix-$J$-unitary
on $\mathcal{J}_d$ if and only if
\begin{itemize}
    \item[(a)] $D$ is $J$-unitary, i.e., $DJD^*=J$;
    \item[(b)] there exists an invertible Hermitian solution
    $H=\diag(H_1,\ldots,H_d)$, with $H_j\in\mat{\mathbb{C}}{m_j}$,
    of the Lyapunov equation
    $$A^*H+HA=-C^*JC,$$
    and
    $$B=-H^{-1}C^*JD.$$
\end{itemize}
This matrix $H$ is uniquely determined by a minimal realization
\eqref{nc-GR}--\eqref{tf-GR}, and for this realization it is
called the \emph{associated structured Hermitian matrix}.
Moreover, $F$ is \emph{matrix-$J$-inner}, i.e., in addition to
\eqref{eq:j-unitary}, $F$ is $J$-contractive on the set of all
$d$-tuples of $n\times n$ matrices $Z_j$ such that $Z_j+Z_j^*>0$,
$n=1,2,\ldots$ (this set is a noncommutative analogue of the right
half-plane), if and only if the associated structured Hermitian
matrix is positive definite.

\section{Difference-differential calculus}
\label{s:difdif} In this section we develop the
difference-differential calculus for noncommutative rational
functions and discuss various special cases and applications:
directional derivatives, backward shifts, finite difference
formulae, higher order difference-differential operators, and
connections with formal power series.

The difference-differential calculus for noncommutative rational
functions is based on difference-differential operators,
$$\Delta_{j}\colon \skfield\to\skfield\otimes\skfield,\quad
j=1,\ldots,d,$$  which are noncommutative counterparts of both
partial finite difference and partial differential operators; here
$\skfield=\field\plangle z_1,\ldots,z_d\prangle$. We extend
$\Delta_{j}$ to matrix-valued noncommutative rational functions by
applying these operators entrywise; we remind the reader that a
matrix-valued noncommutative rational function is the same as a
matrix of (scalar) noncommutative rational functions. We thus have
$$\Delta_{j}\colon \rmat{\skfield}{p}{q}\to\rmat{(\skfield\otimes\skfield)}{p}{q},\quad
j=1,\ldots,d.$$  Our strategy will be to define $\Delta_{j}$ on
matrix-valued noncommutative rational expressions recursively,
starting with matrix-valued noncommutative polynomials, and
postulating linearity and an appropriate version of the Leibniz
rule. We then check that equivalence is preserved, thus we can
define $\Delta_{j}$  on matrix-valued noncommutative rational
functions.

To define $\Delta_{j}$  on matrix-valued noncommutative rational
expressions, we will need to introduce matrix-valued
noncommutative rational expressions in two tuples of noncommuting
indeterminates, $z_1$, \ldots, $z_d$ and $z'_1$, \ldots, $z'_d$.
They are obtained by applying successive matrix arithmetic
operations to tensor products of matrix-valued noncommutative
rational expressions in $z_1$, \ldots, $z_d$ and in $z'_1$,
\ldots, $z'_d$ and forming block matrices. More precisely,
\begin{dfn}\label{dfn:ncexpr2}
\begin{enumerate}
\item \label{tensorprod} If $R$ and $R'$ are $p\times q$ and
$p'\times q'$ matrix-valued noncommutative rational expressions in
$z_1$, \ldots, $z_d$ and in $z'_1$, \ldots, $z'_d$, respectively,
then $R\otimes R'$ is a $pp'\times qq'$ matrix-valued
noncommutative rational expression in $z_1$, \ldots, $z_d$ and
$z'_1$, \ldots, $z'_d$, with $\dom (R\otimes R')=\dom R\times \dom
R'$, and the evaluation is defined by
$$(R\otimes R')(Z,Z')=R(Z)\otimes R'(Z').$$
Here for $Z\in\mattuple{\field}{n}{d}$ and
$Z'\in\mattuple{\field}{n'}{d}$ we have
$R(Z)\in\rmat{\field}{p}{q}\otimes\mat{\field}{n}$,
$R'(Z')\in\rmat{\field}{p'}{q'}\otimes\mat{\field}{n'}$ and
$R(Z)\otimes
R'(Z')\in\rmat{\field}{pp'}{qq'}\otimes\mat{\field}{n}\otimes\mat{\field}{n'}$
via the canonical identification of
$\rmat{\field}{p}{q}\otimes\rmat{\field}{p'}{q'}$ with
$\rmat{\field}{pp'}{qq'}$.
 \item \label{formsum} If $R_1$ and $R_2$ are $p \times q$
matrix-valued noncommutative rational expressions in two tuples of
indeterminates, then so is $R_1 + R_2$, $\dom  (R_1+R_2) = \dom
R_1 \cap \dom R_2$,  and the evaluation is given by
$$(R_1+R_2)(Z,Z') = R_1(Z,Z') + R_2(Z,Z').$$
\item \label{formprod} If $R_1$ and $R_2$ are $p \times q$ and  $q
\times r$ matrix-valued noncommutative rational expressions in two
tuples of indeterminates, then $R_1 R_2$ is $p \times r$
matrix-valued, $\dom (R_1 R_2) = \dom R_1 \cap \dom R_2$,
 and the
evaluation is given by  $$(R_1 R_2)(Z,Z') = R_1(Z,Z') R_2(Z,Z').$$
\item \label{forminv} If $R$ is a $p \times p$ matrix-valued
noncommutative rational expression in two tuples of
indeterminates,
 and $\det R(Z,Z')$ does not vanish identically on $\dom R$,
 then so is $R^{-1}$,
$$
\dom R^{-1} = \left\{ (Z,Z') \in \dom R \colon \det R(Z,Z') \neq
0\right\},
$$
and  $$R^{-1}(Z,Z') = R(Z,Z')^{-1}.$$ \item \label{formmatrix} If
$R_{ab}$, $a=1,\ldots,p_2$, $b=1,\ldots,q_2$, are $p_1 \times q_1$
matrix-valued noncommutative rational expressions in two tuples of
indeterminates, then $R =
\left[R_{ab}\right]_{a=1,\ldots,p_2;\,b=1,\ldots,q_2}$ is  $p_1
p_2 \times q_1 q_2$ matrix-valued,
$$
\dom R = \bigcap_{a=1,\ldots,p_2;\,b=1,\ldots,q_2} \dom R_{ab}$$
and $$ R(Z,Z') =
\left[R_{ab}(Z,Z')\right]_{a=1,\ldots,p_2;\,b=1,\ldots,q_2}.
$$
\end{enumerate}
\end{dfn}

We notice that for a $p\times q$  matrix-valued noncommutative
rational expression $R$ in two tuples of indeterminates and for
$Z\in\mattuple{\field}{n}{d}$ and $Z'\in\mattuple{\field}{n'}{d}$
with $(Z,Z')\in\dom R$, the evaluation
$R(Z,Z')\in\rmat{\field}{p}{q}\otimes\mat{\field}{n}\otimes\mat{\field}{n'}$.
We will often use the canonical identification of
$\rmat{\field}{p}{q}\otimes\mat{\field}{n}\otimes\mat{\field}{n'}=
\rmat{\field}{p}{q}\otimes\left(\mat{\field}{n}\otimes\mat{\field}{n'}\right)$
with $\rmat{\field}{pnn'}{qnn'}$ (cf. page \pageref{page:canon}).
Thus we will often view $R(Z,Z')$ as a $pnn'\times qnn'$ matrix.
(An alternative interpretation of the values  $R(Z,Z')$ as linear
mappings will be considered later --- see the discussion preceding
Theorem \ref{thm:rtriangle}.)

Two $p\times q$ matrix-valued noncommutative rational expressions,
$R_1$ and $R_2$, in two tuples of indeterminates are called
\emph{equivalent} if $\dom R_1\cap\dom R_2\neq \emptyset$ and
$R_1(Z,Z')=R_2(Z,Z')$ for all pairs of $d$-tuples $(Z,Z')$ in
$\dom R_1\cap\dom R_2$. It would be natural to define
matrix-valued noncommutative rational functions in two tuples of
noncommuting indeterminates as the corresponding equivalence
classes. Doing this in a meaningful way requires analogues of
Propositions \ref{prop:nc-invert} and \ref{prop:univer}; see
\cite{Co97,Co98} for related issues. Here we restrict ourselves to
a relatively simple situation.
\begin{thm}\label{thm:ncrat2}
Equivalence classes of $p\times q$ matrix-valued noncommutative
rational expressions in two tuples of indeterminates, which are
formed by using only the rules (1), (2), (3), and (5) in
Definition \ref{dfn:ncexpr2}, are in a natural one-to-one
correspondence with $p\times q$ matrices over
$\skfield\otimes\skfield$.
\end{thm}
\begin{proof}
Since a matrix-valued noncommutative rational expression is
equivalent to a matrix of scalar noncommutative rational
expressions, it is clear that any matrix-valued noncommutative
rational expression in two tuples of indeterminates which is
formed by using only the rules (1), (2), (3), and (5), is
equivalent to a matrix whose entries are sums of tensor products
of noncommutative rational expressions. It only remains to show
that the corresponding elements of $\skfield\otimes\skfield$ are
uniquely determined. Let $\mathfrak{r}_1$, \ldots,
$\mathfrak{r}_\ell$ and $\mathfrak{r}'_1$, \ldots,
$\mathfrak{r}'_\ell$ be noncommutative rational functions in
$z_1$, \ldots, $z_d$ and in $z_1'$, \ldots, $z_d'$ represented by
noncommutative rational expressions $r_1$, \ldots, $r_\ell$ and
$r_1'$, \ldots, $r'_\ell$, respectively. We have to show that if
$r_1\otimes r_1'+\cdots +r_\ell\otimes r'_\ell$ is equivalent to
zero then $\mathfrak{r}_1\otimes \mathfrak{r}_1'+\cdots
+\mathfrak{r}_\ell\otimes \mathfrak{r}'_\ell=0$ in
$\skfield\otimes\skfield$. We may assume that $\mathfrak{r}_1$,
\ldots, $\mathfrak{r}_\ell$ are linearly independent over
$\field$, since otherwise the number of terms in the tensor
combination can be reduced by one.

We may assume that $r'_1$, \ldots, $r'_\ell$ are not all
equivalent to zero, since otherwise there is nothing to prove.
Take $Z'\in\dom r'_1\cap\cdots\cap\dom r'_\ell$ such that
$r'_1(Z')$, \ldots, $r'_\ell(Z')$ are not all zero. (The existence
of such a $Z'$ is established analogously to the reasoning in the
footnote on page \pageref{page:reason}.) This implies that the
matrix elements $(r'_1(Z'))_{ij}$, \ldots, $(r'_\ell(Z'))_{ij}$
are not all zero for some $i$ and $j$. For an arbitrary $Z\in\dom
r_1\cap\cdots\cap\dom r_\ell$, we have $r_1(Z)\otimes
r'_1(Z')+\cdots +r_\ell(Z)\otimes r'_\ell(Z')=0$, and therefore
$$(r'_1(Z'))_{ij}r_1(Z)+\cdots +(r'_\ell(Z'))_{ij}r_\ell(Z)=0$$ is
a nontrivial linear dependance relation for matrices $r_1(Z)$,
\ldots, $r_\ell(Z)$. Therefore $\mathfrak{r}_1$, \ldots,
$\mathfrak{r}_\ell$ are linearly dependent, a contradiction.
\end{proof}
We define the {\em domain of regularity, $\dom \mathfrak{R}$, of a
matrix ${\mathfrak R}$ over $\skfield\otimes\skfield$} as the
union of the domains of regularity of all matrix-valued
noncommutative rational expressions in two tuples of
indeterminates representing $\mathfrak{R}$.

We can also evaluate a matrix-valued noncommutative rational
expression $R$ in two tuples of indeterminates on generic matrices
$T_1$, \ldots, $T_d$ and $T'_1$, \ldots, $T'_d$,
 as in the proof of Proposition~\ref{prop:nc-invert}, and
introduce a subset $\mathcal{N}_R\subseteq\mathbb{N}$ where the
evaluation is defined and the extended domain, $\edom R$.  We then
define $\mathcal{N}_{\mathfrak{R}}$ and $\edom{{\mathfrak R}}$,
the {\em extended domain of regularity of a matrix ${\mathfrak R}$
over $\skfield\otimes\skfield$}, by taking the union over all
matrix-valued noncommutative rational expressions in two tuples of
indeterminates representing $\mathfrak{R}$.
\begin{rem}\label{rem:ncexprell}
We can also introduce matrix-valued noncommutative rational
expressions in $\ell$ tuples of noncommuting indeterminates
$z_1^{(j)}$, \ldots, $z_d^{(j)}$, $j=1$, \ldots, $\ell$,
analogously to Definition \ref{dfn:ncexpr2}, except that $R$ and
$R'$ in rule (1) are now matrix-valued noncommutative rational
expressions in $t$ tuples and in $s$ tuples of indeterminates
respectively, with $t+s=\ell$. Namely, if $R$ and $R'$ are
$p\times q$ and $p'\times q'$ matrix-valued noncommutative
rational expressions in $z_1^{(j)}$, \ldots, $z_d^{(j)}$, $j=1$,
\ldots, $t$, and in $z_1^{(j)}$, \ldots, $z_d^{(j)}$, $j=t+1$,
\ldots, $\ell$, respectively, then $R\otimes R'$ is a $pp'\times
qq'$ matrix-valued noncommutative rational expression in
$z_1^{(j)}$, \ldots, $z_d^{(j)}$, $j=1$, \ldots, $\ell$, with
$\dom (R\otimes R')=\dom R\times \dom R'$, and the evaluation is
defined by
$$(R\otimes R')(Z^{(1)},\ldots,Z^{(\ell)})=R(Z^{(1)},\ldots,Z^{(t)})\otimes
R'(Z^{(t+1)},\ldots, Z^{(\ell)}).$$ We notice that for a $p\times
q$  matrix-valued noncommutative rational expression $R$ in $\ell$
tuples of indeterminates and for
$Z^{(j)}\in\mattuple{\field}{n_j}{d}$, $j=1$, \ldots, $\ell$, with
$(Z^{(1)},\ldots,Z^{(\ell)})\in\dom R$, the evaluation
$$R(Z^{(1)},\ldots,Z^{(\ell)})\in\rmat{\field}{p}{q}\otimes\mat{\field}{n_1}\otimes\cdots\otimes\mat{\field}{n_\ell}.$$
We will often use the canonical identification of
\begin{multline*}
\rmat{\field}{p}{q}\otimes\mat{\field}{n_1}\otimes\cdots\otimes\mat{\field}{n_{\ell-1}}\otimes\mat{\field}{n_\ell}\\=
\rmat{\field}{p}{q}\otimes
\left(\mat{\field}{n_1}\otimes(\cdots\otimes(\mat{\field}{n_{\ell-1}}\otimes\mat{\field}{n_\ell})\cdots)\right)
\end{multline*}
with $\rmat{\field}{pn_1\cdots n_\ell}{qn_1\cdots n_\ell}$.  Thus
we will often view $R(Z^{(1)},\ldots,Z^{(\ell)})$ as a $pn_1\cdots
n_\ell\times qn_1\cdots n_\ell$ matrix. We then define the
equivalence of matrix-valued noncommutative rational expressions
in $\ell$ tuples of indeterminates and show, as in Theorem
\ref{thm:ncrat2}, that equivalence classes of $p\times q$
matrix-valued noncommutative rational expressions in $\ell$ tuples
of indeterminates, which are formed by using only the analogues of
the rules (1), (2), (3), and (5) in Definition \ref{dfn:ncexpr2},
are in a natural one-to-one correspondence with $p\times q$
matrices over $\skfield^{\otimes\ell}$. We can now define the
\emph{domain of regularity, $\dom\mathfrak{R}$, of a matrix
$\mathfrak{R}$ over $\skfield^{\otimes \ell}$} as the union of the
domains of regularity of all matrix-valued noncommutative rational
expressions in $\ell$ tuples of indeterminates representing
$\mathfrak{R}$. We can also introduce a subset
$\mathcal{N}_R\subseteq\mathbb{N}$ where the evaluation on generic
matrices of a matrix-valued noncommutative rational expression $R$
in $\ell$ tuples of indeterminates is defined and the extended
domain, $\edom R$; we then define $\mathcal{N}_\mathfrak{R}$ and
$\edom\mathfrak{R}$, the \emph{extended domain of regularity of a
matrix $\mathfrak{R}$ over $\skfield^{\otimes \ell}$.}
\end{rem}

We proceed now with the definition of difference-differential
operators $\Delta_j$. For a $p\times q$ matrix-valued
noncommutative rational expression $R$, $\Delta_j(R)$ is a
$p\times q$ matrix-valued noncommutative rational expression in
two tuples of indeterminates.
\begin{dfn}\label{dfn:rdif}
\begin{enumerate}
\item  For a matrix-valued noncommutative polynomial $P \in
\rmat{\field}{p}{q} \langle z_1,\ldots,z_d \rangle$,
$P(z)=\sum_{w\in\free_d}P_wz^w$, set
\begin{equation*}
\Delta_j(P)=\sum_{w\in\free_d}P_{w}\sum_{u,v\colon
w=ug_jv}z^u\otimes z^{\prime v} \in \rmat{\left(\field \langle
z_1,\ldots,z_d \rangle\otimes\field \langle z'_1,\ldots,z'_d
\rangle\right)}{p}{q}.
\end{equation*}
\item  If $R_1$ and $R_2$ are $p \times q$ matrix-valued
noncommutative rational expressions, then
\begin{equation*}
\Delta_j(R_1+R_2)=\Delta_j(R_1)+\Delta_j(R_2).
\end{equation*}
\item  If $R_1$ is a $p \times q$ matrix-valued noncommutative
rational expression and $R_2$ is a $q \times s$ matrix-valued
noncommutative rational expression, then
\begin{equation*}
\Delta_j(R_1R_2)=\Delta_j(R_1)(1\otimes R_2)+(R_1\otimes
1)\Delta_j(R_2).
\end{equation*}
\item  If $R$ is a $p \times p$ matrix-valued noncommutative
rational expression which is not identically singular, then
\begin{equation*}
\Delta_j(R^{-1})=-(R^{-1}\otimes 1)\Delta_j(R)(1\otimes R^{-1}).
\end{equation*}
 \item  If $R_{ab}$, $a=1,\ldots,p_2$,
$b=1,\ldots,q_2$, are $p_1 \times q_1$ matrix-valued
noncommutative rational expressions and $R =
\left[R_{ab}\right]_{a=1,\ldots,p_2;\,b=1,\ldots,q_2}$, then
$\Delta_j(R)=\left[(\Delta_j(R))_{ab}\right]_{a=1,\ldots,p_2;\,b=1,\ldots,q_2}$,
\begin{equation*}
(\Delta_j(R))_{ab}=\Delta_j(R_{ab}).
\end{equation*}
\end{enumerate}
\end{dfn}

It is clear that $\dom \Delta_j(R)=\dom R\times\dom R$.

We give some examples to illustrate Definition \ref{dfn:rdif}.
\begin{ex}\label{ex:poly}
For a (scalar) noncommutative polynomial of total degree two in
two indeterminates
\begin{displaymath}
 p=\alpha + \beta z_1 +
\gamma z_2 + \delta z_1^2 + \epsilon z_1 z_2 + \zeta z_2 z_1 +
\eta z_2^2,
\end{displaymath}
we have
$$\Delta_{1}(p)=\beta (1\otimes 1)+\delta
(1\otimes z_1') +\delta (z_1\otimes 1)+\epsilon (1\otimes
z_2')+\zeta (z_2\otimes 1)$$ and
$$\Delta_{2}(p)=\gamma (1\otimes 1)+\epsilon (z_1\otimes 1)+\zeta
(1\otimes z_1')+\eta (1\otimes z_2')+\eta (z_2\otimes 1).$$
\end{ex}
\begin{ex}\label{ex:rat}
For the three equivalent $1\times 1$ matrix-valued rational
expressions $R_1$, $r_2$, and $r_3$ introduced on page
\pageref{three_rats} we have (up to trivial equivalences)
\begin{multline*}
\Delta_{1}(R_1)\sim\left(\begin{bmatrix} 1 & 0
\end{bmatrix}\otimes 1\right) \left(\begin{bmatrix} 1-z_1 & -z_2 \\ -z_2 &
1-z_1\end{bmatrix}^{-1}\otimes
1\right)\\
\cdot\left(1\otimes\begin{bmatrix} 1-z'_1 & -z'_2 \\ -z'_2 &
1-z'_1\end{bmatrix}^{-1}\right)\left(1\otimes
\begin{bmatrix} 1 \\ 0 \end{bmatrix}\right),
\end{multline*}
\begin{multline*}
\Delta_{2}(R_1)\sim\left(\begin{bmatrix} 1 & 0
\end{bmatrix}\otimes 1\right) \left(\begin{bmatrix} 1-z_1 & -z_2 \\ -z_2 &
1-z_1\end{bmatrix}^{-1}\otimes
1\right)\cdot\begin{bmatrix} 0 & 1\\
1 & 0
\end{bmatrix}(1\otimes 1)\\
\cdot \left(1\otimes\begin{bmatrix} 1-z'_1 & -z'_2 \\ -z'_2 &
1-z'_1\end{bmatrix}^{-1}\right)\left( 1\otimes \begin{bmatrix} 1 \\
0
\end{bmatrix}\right),
\end{multline*}
\begin{multline*}
\Delta_{1}(r_2)\sim\left((1-z_1-z_2 (1-z_1)^{-1}z_2)^{-1}\otimes
1\right)\\
\cdot\left( 1\otimes 1+(z_2(1-z_1)^{-1}\otimes 1)(1\otimes
(1-z_1')^{-1}z_2')\right)\\
\mbox{\hspace{5cm}}\cdot\left(1\otimes(1-z_1'-z_2' (1-z_1')^{-1}z_2')^{-1}\right),\\
\end{multline*}
\begin{multline*}
\Delta_{2}(r_2)\sim\left((1-z_1-z_2 (1-z_1)^{-1}z_2)^{-1}\otimes
1\right)\\
\cdot\left( 1\otimes (1-z_1')^{-1}z_2'+z_2(1-z_1)^{-1}\otimes
1\right)\\
\mbox{\hspace{5cm}}\cdot\left(1\otimes(1-z_1'-z_2'
(1-z_1')^{-1}z_2')^{-1}\right),
\end{multline*}
\begin{multline*}
\Delta_{1}(r_3)\sim(z_2^{-1}\otimes
1)\left(1\otimes(z'_2-(1-z'_1){z'_2}^{-1}(1-z'_1))^{-1}\right)\\
+\left(z_2^{-1}(1-z_1)\otimes
1\right)\left((z_2-(1-z_1)z_2^{-1}(1-z_1))^{-1}\otimes 1\right)\\
\cdot\left(1\otimes {z_2'}^{-1}(1-z_1')+(1-z_1)z_2^{-1}\otimes 1\right)\\
\cdot\left(1\otimes(z_2'-(1-z'_1){z'_2}^{-1}(1-z'_1))^{-1}\right),
\end{multline*}
\begin{multline*}
\Delta_{2}(r_3)\sim\left(z_2^{-1}\otimes 1\right)\left(1\otimes
{z_2'}^{-1}\right)
\left(1\otimes(1-z_1')(z'_2-(1-z'_1){z'_2}^{-1}(1-z'_1))^{-1}\right)\\
+\left(z_2^{-1}(1-z_1)\otimes
1\right)\left((z_2-(1-z_1)z_2^{-1}(1-z_1))^{-1}\otimes 1\right)\\
\cdot\left(1\otimes 1+\left((1-z_1)z_2^{-1}\otimes
1\right)\left(1\otimes  {z_2'}^{-1}(1-z_1')\right)\right)\\
\cdot\left(1\otimes(z_2'-(1-z'_1){z'_2}^{-1}(1-z'_1))^{-1}\right).
\end{multline*}
\end{ex}
\begin{ex}\label{ex:realiz}
Consider a $p\times q$ matrix-valued noncommutative rational
expression
\begin{equation}\label{eq:rsr}
R=C(I_m-A_1z_1-\cdots -A_dz_d)^{-1}B,
\end{equation}
where $m\in\mathbb{N}$, $A_j\in\mat{\field}{m}$ ($j=1,\ldots,d$),
$B\in\rmat{\field}{m}{q}$, $C\in\rmat{\field}{p}{m}$. (This is a
recognizable series realization of a $p\times q$ matrix-valued
noncommutative rational function; see
\cite{Kle,Schutz61,Fliess74a}.) We have
\begin{multline*}
\Delta_j(R)=(C\otimes 1)\cdot\left((I_m-A_1z_1-\cdots
-A_dz_d)^{-1}\otimes
1\right)\cdot A_j(1\otimes 1)\\
\cdot\left( 1\otimes (I_m-A_1z'_1-\cdots
-A_dz'_d)^{-1}\right)\cdot(1\otimes B),
\end{multline*}
or up to trivial equivalences,
\begin{multline}\label{eq:rsr_dif-lr}
\Delta_j(R)\sim\left(C(I_m-A_1z_1-\cdots -A_dz_d)^{-1}A_j\otimes
1\right)\\
\cdot\left( 1\otimes (I_m-A_1z'_1-\cdots -A_dz'_d)^{-1}B\right)\\
\sim\left(C(I_m-A_1z_1-\cdots -A_dz_d)^{-1}\otimes 1\right)
\cdot\left( 1\otimes A_j(I_m-A_1z'_1-\cdots
-A_dz'_d)^{-1}B\right).
\end{multline}
\end{ex}

It is not \emph{a priori} clear from Definition \ref{dfn:rdif}
that $\Delta_j$ preserves the equivalence of matrix-valued
noncommutative rational expressions and can be thus defined on
matrix-valued noncommutative rational functions. This is a
consequence of the following key theorem that relates the
evaluation $\Delta_j(R)(Z,Z')$ to the evaluation of $R$ on
$d$-tuples of $2\times 2$ block upper triangular matrices with
$Z_j$ and $Z'_j$ on block diagonals. We will identify
$\mat{\field}{n} \otimes \mat{\field}{n'}$ with
$\hom(\rmat{\field}{n}{n'}, \rmat{\field}{n}{n'})$, so that
$\sum_i A_i \otimes B_i$ corresponds to the linear mapping $H
\mapsto \sum_i A_i H B_i$. This correspondence extends naturally
to $p\times q$ matrices: we identify
$\rmat{\field}{p}{q}\otimes\mat{\field}{n} \otimes
\mat{\field}{n'}\cong\rmat{(\mat{\field}{n} \otimes
\mat{\field}{n'})}{p}{q}$ with $\hom\left(\rmat{\field}{n}{n'},
\rmat{(\rmat{\field}{n}{n'})}{p}{q}\right)$, so that $\left[\sum_i
A^{(jk)}_i \otimes
B^{(jk)}_i\right]_{j=1,\ldots,p;\,k=1,\ldots,q}$ corresponds to
the linear mapping $H \mapsto \left[\sum_i A^{(jk)}_i H
B^{(jk)}_i\right]_{j=1,\ldots,p;\,k=1,\ldots,q}$. We will use the
$pn\times pn$ permutation matrices
$P(p,n)=[E_{ij}^T]_{i=1,\ldots,p;\,j=1,\ldots,n}$ where each
$E_{ij}\in\rmat{\field}{p}{n}$ has entry $1$ in position $(i,j)$
and all other entries are zero. These matrices allow us to change
the order of factors in tensor products: $A\otimes B=P(n,p)
(B\otimes A)P(m,q)^{\mtrans}$ for any $A\in\rmat{\field}{n}{m}$
and $B\in\rmat{\field}{p}{q}$. See \cite[pages 259--261]{HJ}. We
also use the notation $\dom_nR:=\dom
R\cap\mattuple{\field}{n}{d}$.

\begin{thm}\label{thm:rtriangle}
Let $R$ be a $p\times q$ matrix-valued noncommutative rational
expression. Let $Z\in\dom_n R$ and $Z'\in\dom_{n'} R$, and let
$W=(W_1,\ldots,W_d)\in \rmattuple{\field}{n}{n'}{d}$. Then
$$\begin{bmatrix} Z & W\\ 0 & Z'
\end{bmatrix}=\left(\begin{bmatrix} Z_1 & W_1\\
0 & Z'_1
\end{bmatrix},\ldots,\begin{bmatrix} Z_d & W_d\\
0 & Z'_d
\end{bmatrix}\right)\in\dom_{n+n'} R$$ and
\begin{multline}\label{eq:rtriangle}
P(n+n',p)R\left(\begin{bmatrix} Z & W\\ 0 & Z'
\end{bmatrix}\right)P(n+n',q)^{\mtrans} \\
 =  \begin{bmatrix}
P(n,p)R(Z)P(n,q)^{\mtrans} &
P(n,p)\left(\sum\limits_{j=1}^d\Delta_j(R)(Z,Z')(W_j)\right)P(n',q)^{\mtrans}\\
0 & P(n',p)R(Z')P(n',q)^{\mtrans}
\end{bmatrix}.
\end{multline}
\end{thm}
\begin{proof}
We establish \eqref{eq:rtriangle} recursively; the reasoning will
also imply that
$\begin{bmatrix} Z & W\\
0 & Z'
\end{bmatrix}\in\dom_{n+n'}R$.

We first verify that \eqref{eq:rtriangle} holds when $R$ is a
matrix-valued noncommutative polynomial
$P=\sum_{w\in\free_d}P_wz^w$.
\begin{multline*}
P(n+n',p)P\left(\begin{bmatrix} Z & W\\ 0 & Z'
\end{bmatrix}\right)P(n+n',q)^{\mtrans} =\sum_{w\in\free_d}\begin{bmatrix} Z & W\\ 0 & Z'
\end{bmatrix}^w\otimes P_w\\
=\sum_{\ell\in\mathbb{Z}_+,\,i_1,\ldots,i_\ell\in\{1,\ldots,
d\}}\begin{bmatrix} Z_{i_1} & W_{i_1}\\
0 & Z'_{i_1}
\end{bmatrix}\cdots\begin{bmatrix} Z_{i_\ell} & W_{i_\ell}\\
0 & Z'_{i_\ell}
\end{bmatrix}\otimes P_{g_{i_1}\cdots g_{i_\ell}}\\
=\sum_{\ell,\,i_1,\ldots,i_\ell}\begin{bmatrix} Z_{i_1}\cdots
Z_{i_\ell} & \sum\limits_{k=1}^\ell Z_{i_1}\cdots
Z_{i_{k-1}}W_{i_k}Z'_{i_{k+1}}\cdots Z'_{i_\ell}
\\ 0
& Z'_{i_1}\cdots Z'_{i_\ell}
\end{bmatrix}\otimes P_{g_{i_1}\cdots
g_{i_\ell}}\\
=\small\begin{bmatrix}
\sum\limits_{\ell,\,i_1,\ldots,i_\ell}Z_{i_1}\cdots
Z_{i_\ell}\otimes P_{g_{i_1}\cdots g_{i_\ell}} &
\sum\limits_{\ell,\,i_1,\ldots,i_\ell}\sum\limits\limits_{k=1}^\ell
Z_{i_1}\cdots Z_{i_{k-1}}W_{i_k}Z'_{i_{k+1}}\cdots
Z'_{i_\ell}\otimes P_{g_{i_1}\cdots g_{i_\ell}}
\\ 0
& \sum\limits_{\ell,\,i_1,\ldots,i_\ell}Z'_{i_1}\cdots
Z'_{i_\ell}\otimes P_{g_{i_1}\cdots g_{i_\ell}}
\end{bmatrix}\\
=\begin{bmatrix} \sum\limits_wZ^w\otimes P_w &
\sum\limits_w\sum\limits_{j=1}^d\,\sum\limits_{u,v\colon
w=ug_jv}Z^uW_jZ'^v\otimes P_w\\
0 & \sum\limits_wZ'^w\otimes P_w
\end{bmatrix} \\
=\begin{bmatrix} P(n,p)P(Z)P(n,q)^{\mtrans} &
P(n,p)\left(\sum\limits_{j=1}^d\Delta_j(P)(Z,Z')(W_j)\right)P(n',q)^{\mtrans}
\\ 0
&  P(n',p)P(Z')P(n',q)^{\mtrans}
\end{bmatrix}.
\end{multline*}

If \eqref{eq:rtriangle} is true for $R_1$ and for $R_2$ then it is
clearly true for $R=R_1+R_2$.

Now, assume that \eqref{eq:rtriangle} is true for a $p\times q$
matrix-valued noncommutative rational expression $R_1$ and for a
$q\times r$ matrix-valued noncommutative rational expression
$R_2$. Then
\begin{multline*}
P(n+n',p)(R_1R_2)\left(\begin{bmatrix} Z & W\\ 0 & Z'
\end{bmatrix}\right)P(n+n',r)^{\mtrans}\\
=P(n+n',p)R_1\left(\begin{bmatrix} Z & W\\
0 & Z'
\end{bmatrix}\right)P(n+n',q)^{\mtrans}P(n+n',q)R_2\left(\begin{bmatrix} Z & W\\
0 & Z'
\end{bmatrix}\right)P(n+n',r)^{\mtrans}\\
=\begin{bmatrix} P(n,p)R_1(Z)P(n,q)^{\mtrans} &
P(n,p)\left(\sum\limits_{j=1}^d\Delta_j(R_1)(Z,Z')(W_j)\right)P(n',q)^{\mtrans}\\
0 & P(n',p)R_1(Z')P(n',q)^{\mtrans}
\end{bmatrix}\\
\cdot
\begin{bmatrix}
 P(n,q)R_2(Z)P(n,r)^{\mtrans} &
 P(n,q)\left(\sum\limits_{j=1}^d\Delta_j(R_2)(Z,Z')(W_j)\right)P(n',r)^{\mtrans}\\
0 & P(n',q)R_2(Z')P(n',r)^{\mtrans}
\end{bmatrix}\\
 = \left[\begin{matrix}
P(n,p)R_1(Z)R_2(Z)P(n,r)^{\mtrans}\\
 0
 \end{matrix}\right.\mbox{\hspace{10cm}} \\
\left.\begin{matrix} P(n,p)
\sum\limits_{j=1}^d\left(R_1(Z)\Delta_j(R_2)(Z,Z')(W_j)+
\Delta_j(R_1)(Z,Z')(W_j)R_2(Z')\right)P(n',r)^{\mtrans}\\
 P(n',p)R_1(Z')R_2(Z')P(n',r)^{\mtrans}
\end{matrix}\right]
\\
=\begin{bmatrix} P(n,p)(R_1R_2)(Z)P(n,r)^{\mtrans} &
 P(n,p)\left(\sum\limits_{j=1}^d\Delta_j(R_1R_2)(Z,Z')(W_j)\right)P(n',r)^{\mtrans}\\
0 & P(n',p)(R_1R_2)(Z')P(n',r)^{\mtrans}
\end{bmatrix}.
\end{multline*}
 Thus, \eqref{eq:rtriangle} is true for the
$p\times r$ matrix-valued noncommutative rational expression
$R_1R_2$.

Next, assume that \eqref{eq:rtriangle} is true for a $p\times p$
matrix-valued noncommutative rational expression $R$  which is not
identically singular. Then
\begin{multline*}
P(n+n',p)R^{-1}\left(\begin{bmatrix} Z & W\\ 0 & Z'
\end{bmatrix}\right)P(n+n',p)^{\mtrans}\\
= \left(P(n+n',p)R\left(\begin{bmatrix} Z & W\\
0 & Z'
\end{bmatrix}\right)P(n+n',p)^{\mtrans}\right)^{-1}\\
= \begin{bmatrix}
P(n,p)R(Z)P(n,p)^{\mtrans} &
P(n,p)\left(\sum\limits_{j=1}^d\Delta_j(R)(Z,Z')(W_j)\right)P(n',p)^{\mtrans}\\
0 & P(n',p)R(Z')P(n',p)^{\mtrans}
\end{bmatrix}^{-1}
\\
=\left[\begin{matrix} P(n,p)R(Z)^{-1}P(n,p)^{\mtrans}\\
0 \end{matrix}\right.\hfill \\
\left.\begin{matrix} -P(n,p)R(Z)^{-1}\left(
\sum\limits_{j=1}^d\Delta_j(R)(Z,Z')(W_j)\right)R(Z')^{-1}P(n',p)^{\mtrans}\\
P(n',p)R(Z')^{-1}P(n',p)^{\mtrans}
\end{matrix}\right]\\
=\begin{bmatrix} P(n,p)R^{-1}(Z)P(n,p)^{\mtrans} &
 P(n,p)\left(\sum\limits_{j=1}^d\Delta_j(R^{-1})(Z,Z')(W_j)\right)P(n',p)^{\mtrans}\\
0 & P(n',p)R^{-1}(Z')P(n',p)^{\mtrans}
\end{bmatrix}.
\end{multline*}
Thus, \eqref{eq:rtriangle} is true for the matrix-valued
noncommutative rational expression $R^{-1}$.

Finally, we will show that if \eqref{eq:rtriangle} is true for
$R_{ab},\ a=1,\ldots,p_2,\ b=1,\ldots, q_2$, then it is also true
for $R=[R_{ab}]$. Clearly, it suffices to prove
\eqref{eq:rtriangle}  for the case where only one block $R_{ab}$
is nonzero, i.e., $R=E_{ab}\otimes R_{ab}$, with some $a,b$ and
$E_{ab}$ a $p_2\times q_2$ matrix with $1$ at the $(a,b)$ position
and $0$ elsewhere. Since $(E_{ab}\otimes R_{ab})(X)=(E_{ab}\otimes
I_{p_1})\cdot(I_{q_2}\otimes R_{ab})(X)$, this boils down to
proving \eqref{eq:rtriangle} for $I_{q_2}\otimes R_{ab}$.
Simplifying the notation, we can state the problem as follows:
show that if a $p\times q$ matrix-valued noncommutative rational
expression $R$ satisfies \eqref{eq:rtriangle}, then so does the
$sp\times sq$ matrix-valued noncommutative rational expression
$I_{s}\otimes R$ for every $s\in\mathbb{N}$. We first observe that
for any $m,s\in\mathbb{N}$ and any $X\in\dom_mR$ one has $X\otimes
I_s\in\dom_{ms}R$ and $R(X)\otimes I_s=R(X\otimes I_s)$ --- this
follows directly from the definition of matrix-valued
noncommutative rational expressions. Second, using this
observation and the identities of the form
$$P(m,sp)=P(ms,p)P(mp,s)$$
(see \cite[Problem 20, page 266]{HJ}), we obtain that
\begin{multline*}
P(m,sp)(I_s\otimes
R)(X)P(m,sq)^{\mtrans}\\
=P(ms,p)P(mp,s)(I_s\otimes
R(X))P(mq,s)^{\mtrans}P(ms,q)^{\mtrans}\\
=P(ms,p)R(X\otimes I_s)P(ms,q)^{\mtrans}.
\end{multline*}
 Then we have
\begin{multline*}
P(n+n',sp)(I_s\otimes R)\left(\begin{bmatrix} Z & W\\
0 & Z'
\end{bmatrix}\right)P(n+n',sq)^{\mtrans}\\
=P((n+n')s,p)R\left(\begin{bmatrix} Z\otimes I_s & W\otimes I_s\\
0 & Z'\otimes I_s
\end{bmatrix}\right)P((n+n')s,q)^{\mtrans}\\
=\left[ \begin{matrix} P(ns,p)R(Z\otimes I_s)P(ns,q)^{\mtrans}\\
0
\end{matrix}\right.\hfill \\
 \left.\begin{matrix}
P(ns,p)\left(\sum_{j=1}^d\Delta_j(R)(Z\otimes I_s,Z'\otimes
I_s)(W_j\otimes I_s)\right)P(n's,q)^{\mtrans}\\
P(n's,p)R(Z'\otimes I_s)P(n's,q)^{\mtrans}
\end{matrix}\right]\\
=\left[ \begin{matrix} P(n,sp)(I_s\otimes R)(Z)P(n,sq)^{\mtrans}
\\
0
\end{matrix}\right. \hfill \\
\left. \begin{matrix}
P(n,sp)\left(\sum_{j=1}^d\Delta_j(I_s\otimes R)(Z,Z')(W_j)\right)P(n',sq)^{\mtrans}\\
P(n',sp)(I_s\otimes R)(Z')P(n',sq)^{\mtrans}
\end{matrix}\right],
\end{multline*}
i.e., $I_s\otimes R$ satisfies \eqref{eq:rtriangle}.

The proof is complete.
\end{proof}
\begin{rem}\label{rem:corr}
The equality \eqref{eq:rtriangle} is the accurate statement of the
fact that matrix-valued noncommutative rational expressions
respect a block triangular matrix structure, in particular direct
sums --- compare \cite[formulae (2.3), (2.5), (2.12), (2.13),
(2.19)]{KVV}, where the permutation matrices are missing. (This
does not affect any subsequent arguments there.)
\end{rem}
\begin{cor}\label{cor:redom}
For any matrix-valued noncommutative rational expression $R$,
$$\edom \Delta_j(R)\supseteq \edom R\times\edom R.$$
\end{cor}
The proof is obtained by substituting generic matrices $T_1$,
\ldots, $T_d$ and $T'_1$, \ldots, $T_d'$ into \eqref{eq:rtriangle}
-- the details are similar to \cite[Corollary 2.12]{KVV}.

\begin{cor}\label{cor:requiv}
If $R_1$ and $R_2$ are two equivalent matrix-valued rational
expressions then $\Delta_j(R_1)$ and $ \Delta_j(R_2)$ are also
equivalent.
\end{cor}
The proof is immediate from \eqref{eq:rtriangle}.

Corollary \ref{cor:requiv} and Theorem \ref{thm:ncrat2} allow us
to define  $\Delta_j\mathfrak{R}$ for a matrix-valued
noncommutative rational function $\mathfrak{R}$:
$\Delta_j\mathfrak{R}$ is the matrix over
$\skfield\otimes\skfield$ corresponding to the equivalence class
of $\Delta_j(R)$ for any
 $R\in\mathfrak{R}$. We have $$ \dom\Delta_j\mathfrak{R}\supseteq
\dom\mathfrak{R}\times\dom\mathfrak{R},\qquad
\edom\Delta_j\mathfrak{R}\supseteq
\edom\mathfrak{R}\times\edom\mathfrak{R},$$ where the second
inclusion follows from Corollary \ref{cor:redom}.
\begin{rem}\label{rem:doms-inclus}
We emphasize that while we always have equality
$\dom\Delta_j(R)=\dom(R)\times\dom(R)$ for a matrix-valued
noncommutative rational expression $R$, we may have strict
inclusion $\dom\Delta_j\mathfrak{R}\supsetneq
\dom\mathfrak{R}\times\dom\mathfrak{R}$. As an example, let
$r=z_1^{-1}$. Then we have $$\Delta_2r=-(z_1^{-1}\otimes 1)\cdot
0(1\otimes 1)\cdot(1\otimes z_1^{-1})\sim 0,$$ and for the
corresponding noncommutative rational function $\mathfrak{r}$ we
have
$$\dom\Delta_2\mathfrak{r}=\coprod_{n=1}^\infty
\mattuple{\field}{n}{2},$$ while
$$\dom\mathfrak{r}=\{(Z_1,Z_2)\colon\det Z_1\neq 0\}.$$
This example also shows that we have a strict inclusion
$$\edom(\Delta_2r)\supsetneq\edom(r)\times\edom(r).$$
\end{rem}

We proceed to describe some of the many facets of the
noncommutative difference-differential operators.

\subsection{Directional derivatives.}\label{s:dir-der} Evaluating
$\Delta_j(R)$, $j=1,\ldots,d$, at $Z'=Z$ yields the differential
or the directional derivatives of a matrix-valued noncommutative
rational expression $R$ at $Z$:
\begin{equation}\label{eq:dir-der}
\sum_{j=1}^d\Delta_j(R)(Z,Z)(W_j)=\frac{d}{dt}R(Z+tW)\big|_{t=0}.
\end{equation}
Here $Z\in\dom R$ and we view $R(Z+tW)$ as a matrix-valued
rational function in the indeterminate $t$. The formula
\eqref{eq:dir-der} follows easily from the recursive definition of
$\Delta_j$ and the corresponding properties of the derivative of a
matrix-valued rational function in one indeterminate. Of course,
one can also write the analogue of \eqref{eq:dir-der} for
matrix-valued noncommutative rational functions.

\subsection{Backward shifts.}\label{s:bshifts} Let $R$ be a matrix-valued noncommutative
rational expression which is regular at zero, i.e., $0\in\dom_1R$.
Then
\begin{equation}\label{eq:bshifts}
\Delta_j(R)(Z,0)=\rs_j R(Z),\quad \Delta_j(R)(0,Z)=\ls_j R(Z),
\end{equation}
where $\rs_j$ and $\ls_j$ are the right and left backward shift
operators introduced in \cite{KVV}. The formula \eqref{eq:bshifts}
follows easily from the recursive definitions of $\Delta_j$ and of
the backward shifts. Of course, one can also write the analogue of
\eqref{eq:bshifts} for matrix-valued noncommutative rational
functions. As an  illustration of the action of backward shifts,
for the polynomial $p$ of Example \ref{ex:poly} we have
$$\rs_1p=\beta+\delta z_1+\zeta z_2,\quad \ls_1p=\beta +\delta
z_1+\epsilon z_2,$$
$$\rs_2p=\gamma+\epsilon z_1+\eta z_2,\quad \ls_2p=\gamma +\zeta
z_1+\eta z_2,$$ and for the recognizable series realization
\eqref{eq:rsr} of Example \ref{ex:realiz}, we have
$$\rs_jR=C(I_m-A_1z_1-\cdots -A_dz_d)^{-1}A_jB,\quad \ls_jR=CA_j(I_m-A_1z_1-\cdots
-A_dz_d)^{-1}B.$$

\subsection{Finite difference formulae.}\label{s:fin-dif}
For a $p\times q$ matrix-valued noncommutative rational expression
$R$, and for $Z^{(0)},Z\in\dom_nR$, we have the noncommutative
finite difference formula
$$R(Z)-R\left(Z^{(0)}\right)=\sum_{j=1}^d\Delta_j(R)\left(Z^{(0)},Z\right)\left(Z_j-Z^{(0)}_j\right),$$
which can also be proved recursively, and extends naturally to
matrix-valued noncommutative rational functions.

\subsection{Higher order difference-differential
operators.}\label{s:higher} We can iterate the
differ\-ence-differential operators $\Delta_j$: we define a linear
mapping
$$\Delta_j\colon \skfield^{\otimes \ell}\to
\skfield^{\otimes(\ell+1)}$$ by its action on pure tensors as
$$\Delta_j(\mathfrak{r}_1\otimes\cdots\otimes\mathfrak{r}_\ell)=
\mathfrak{r}_1\otimes\cdots\otimes\mathfrak{r}_{\ell-1}\otimes\Delta_j\mathfrak{r}_\ell.$$
It is easy to check that $\Delta_j$ satisfies the following
version of the Leibniz rule:
\begin{equation}\label{eq:Leibniz-ell}
\Delta_j(\mathfrak{r}\mathfrak{r'})=\Delta_j\mathfrak{r}\cdot
\iota_\ell\mathfrak{r}'+(\mathfrak{r}\otimes
1)\cdot\Delta_j\mathfrak{r}',
\end{equation}
for all $\mathfrak{r}$, $\mathfrak{r}'\in\skfield^{\otimes\ell}$,
where the linear mapping $\iota_\ell\colon\skfield^{\otimes
\ell}\to\skfield^{\otimes \ell+1}$ (actually, a homomorphism of
$\field$-algebras) is defined on pure tensors by
$$\iota_\ell(\mathfrak{r}_1\otimes\cdots\otimes\mathfrak{r}_{\ell-1}\otimes\mathfrak{r}_{\ell})=
\mathfrak{r}_1\otimes\cdots\otimes\mathfrak{r}_{\ell-1}\otimes1\otimes\mathfrak{r}_{\ell}.$$
Applying \eqref{eq:Leibniz-ell} to both sides of the identity
$\mathfrak{r}\mathfrak{r}^{-1}=1$, we obtain that
\begin{equation}\label{eq:inv-dif-ell}
\Delta_j(\mathfrak{r}^{-1})=-(\mathfrak{r}^{-1}\otimes
1)\cdot\Delta_j\mathfrak{r}\cdot\iota_\ell(\mathfrak{r}^{-1}).
\end{equation}
We extend $\Delta_j$ and $\iota_\ell$ entrywise to matrices.

 We can now define, for a
word $w=g_{i_\ell}\cdots g_{i_1}$ of length $\ell$, the
corresponding higher-order difference-differential operators
$$\Delta^w:=\Delta_{i_\ell}\cdots\Delta_{i_1}\colon\rmat{\skfield}{p}{q}\to
\rmat{\left(\skfield^{\otimes(\ell+1)}\right)}{p}{q}.$$

\begin{ex}\label{ex:poly-dif_ell}
For a matrix-valued noncommutative rational function
$\mathfrak{R}$ defined by a matrix-valued noncommutative
polynomial $P=\sum_{v\in\free_d}P_vz^v$ and for a word
$w=g_{i_\ell}\cdots g_{i_1}$ of length $\ell$ at most the total
degree of $P$, we have
$$\Delta^w\mathfrak{R}=\sum_{v\in\free_d}P_v\sum_{v=u_1g_{i_1}u_2g_{i_2}\cdots g_{i_\ell}u_{\ell+1}}
(z^{(1)})^{u_1}\otimes (z^{(2)})^{u_2}\otimes\cdots\otimes
(z^{(\ell+1)})^{u_{\ell+1}}.$$ More precisely, every term in the
second sum on the right-hand side is a tensor product of $\ell+1$
noncommutative rational functions defined by the corresponding
noncommutative monomials; alternatively, the right-hand side (with
some nesting of parentheses) is a matrix-valued noncommutative
rational expression in $\ell$ tuples of indeterminates defining
$\Delta^w\mathfrak{R}$.
\end{ex}
\begin{ex}\label{ex:realiz_dif}
For the matrix-valued noncommutative rational function
$\mathfrak{R}$ defined by a recognizable series realization of the
form \eqref{eq:rsr}, we have by iterating the first equality in
\eqref{eq:rsr_dif-lr} that for every $w=g_{i_\ell}\cdots g_{i_1}$,
\begin{multline*}
\Delta^w\mathfrak{R}=\left(C(I_m-A_1z_1^{(1)}-\cdots-A_dz_d^{(1)})^{-1}A_{i_1}\otimes
\underset{\ell\ {\rm times}}{\underbrace{1\otimes\cdots\otimes 1}}\right)\\
\cdot\prod_{j=2}^\ell\left(\underset{j-1\ {\rm
times}}{\underbrace{1\otimes\cdots\otimes
1}}\otimes(I_m-A_1z_1^{(j)}-\cdots-A_dz_d^{(j)})^{-1}A_{i_j}\otimes\underset{\ell-j+1\
{\rm times}}{\underbrace{1\otimes\cdots\otimes 1}}\right)\\
\cdot\left(\underset{\ell\ {\rm
times}}{\underbrace{1\otimes\cdots\otimes
1}}\otimes(I_m-A_1z_1^{(\ell+1)}-\cdots-A_dz_d^{(\ell+1)})^{-1}B\right).
\end{multline*}
\end{ex}

We identify
$\mat{\field}{n_1}\otimes\cdots\otimes\mat{\field}{n_{\ell+1}}$
with $\ell$-linear mappings
$$\rmat{\field}{n_1}{n_2}\times
\cdots\times\rmat{\field}{n_{\ell}}{n_{\ell+1}}\longrightarrow\rmat{\field}{n_1}{n_{\ell+1}},$$
so that $\sum_i\left(A^{(1)}_i\otimes\cdots\otimes
A^{(\ell+1)}_i\right)$ corresponds to the $\ell$-linear mapping
$$(H_1,\ldots,H_\ell)\longmapsto\sum_iA^{(1)}_iH_1A^{(2)}_i\cdots A^{(\ell)}_iH_\ell
A^{(\ell+1)}_i.$$ This correspondence extends naturally to
matrices. It follows that for a $p\times q$ matrix-valued
noncommutative rational function $\mathfrak{R}$ and for
$Z^{(1)}\in\mat{\field}{n_1}$, \ldots,
$Z^{(\ell+1)}\in\mat{\field}{n_{\ell+1}}$ in appropriate domains,
we have that
\begin{multline*}\Delta^w\mathfrak{R}\left(Z^{(1)},\ldots,Z^{(\ell+1)}\right)\in
\rmat{\field}{p}{q}\otimes\mat{\field}{n_1}\otimes\cdots\otimes\mat{\field}{n_{\ell+1}}\\
\cong\rmat{\left(\mat{\field}{n_1}\otimes\cdots\otimes\mat{\field}{n_{\ell+1}}\right)}{p}{q},
\end{multline*}
and for
$(H_1,\ldots,H_\ell)\in\rmat{\field}{n_1}{n_2}\times\cdots\times\rmat{\field}{n_{\ell}}{n_{\ell+1}}$,
we have that
$$
\Delta^w\mathfrak{R}\left(Z^{(1)},\ldots,Z^{(\ell+1)}\right)(H_1,\ldots,H_\ell)\in
\rmat{\left(\rmat{\field}{n_1}{n_{\ell+1}}\right)}{p}{q}.$$

In particular, we have
$$\sum_{i=1}^d\sum_{j=1}^d\Delta_{i}\Delta_j\mathfrak{R}(Z,Z,Z)(W_i,W_j)=
\frac{d^2}{dt^2}\mathfrak{R}(Z+tW)\Big|_{t=0}.$$ This is exactly
the Hessian of $\mathfrak{R}$ which plays a central role in the
study of noncommutative convexity; see, e.g.,
\cite{H03,HMcCV,HMcCPV}.  Hessians and directional derivatives of
matrix-valued noncommutative rational expressions were implemented
in NCAlgebra to produce a convexity checking algorithm
\cite{CHSY}.

\begin{rem}\label{rem:rdifell}
 We can also define the difference-differential
operators on matrices over tensor powers of $\skfield$ at the
level of matrix-valued noncommutative rational expressions in
several tuples of indeterminates. For a $p\times q$ matrix-valued
noncommutative rational expression $R$ in $\ell$ tuples of
indeterminates (see Remark \ref{rem:ncexprell}), we define
$\Delta_j(R)$, a $p\times q$ matrix-valued noncommutative rational
expression in $\ell+1$ tuples of indeterminates, analogously to
Definition \ref{dfn:rdif}, except that in rule (1) we consider
``pure tensors" instead of polynomials, and we modify rules (3)
and (4), cf. \eqref{eq:Leibniz-ell} and \eqref{eq:inv-dif-ell}.
Namely,
\begin{enumerate}
    \item[(1)] If $R$ and $R'$ are matrix-valued noncommutative rational
expressions in $t$ tuples and in $s$ tuples of indeterminates
respectively, with $t+s=\ell$, we set
$$\Delta_j(R\otimes R')=R\otimes \Delta_j(R').$$
    \item[(3)] If $R_1$ and $R_2$ are matrix-valued noncommutative
    rational expressions of compatible sizes in $\ell$ tuples of intederminates, then
    $$\Delta_j(R_1R_2)=\Delta_j(R_1)\iota_\ell(R_2)+(R_1\otimes
    1)\Delta_j(R_2).$$
    \item[(4)] If $R$ is a square matrix-valued noncommutative
    rational expression in $\ell$ tuples of indeterminates, which
    is not identically singular, then
    $$\Delta_j(R^{-1})=-(R^{-1}\otimes
    1)\Delta_j(R)\iota_\ell(R^{-1}).$$
    \end{enumerate}
Here, for a matrix-valued noncommutative rational expression $R$
in $\ell$ tuples of intederminates, $\iota_\ell(R)$ is a
matrix-valued noncommutative rational expression of the same size
in $\ell+1$ tuples of intederminates defined by
$\iota_1(R)=1\otimes R$ and by the recursive relations
\begin{enumerate}
\item $\iota_{t+s}(R\otimes R')=R\otimes \iota_{s}(R')$; \item
$\iota_\ell(R_1+R_2)=\iota_\ell(R_1)+\iota_\ell(R_2)$; \item
$\iota_\ell(R_1R_2)=\iota_\ell(R_1)\iota_\ell(R_2)$; \item
$\iota_\ell(R^{-1})=(\iota_\ell(R))^{-1}$; \item
$\iota_\ell([R_{ab}])=[\iota_\ell(R_{ab})]$
\end{enumerate}
for $\ell>1$. Notice that for $d$-tuples of matrices
$Z^j\in\mattuple{\field}{n_j}{d}$, $j=1$, \ldots, $\ell+1$, the
value $\iota_\ell(R)(Z^{(1)},\ldots,Z^{(\ell+1)})$ is the image of
the value $R(Z^{(1)},\ldots,Z^{(\ell-1)},Z^{(\ell+1)})$ under the
linear mapping
\begin{multline*}
\mat{\field}{n_1}\otimes
\cdots\otimes\mat{\field}{n_{\ell-1}}\otimes\mat{\field}{n_{\ell+1}}\\
\hfill\longrightarrow \mat{\field}{n_1}\otimes
\cdots\otimes\mat{\field}{n_{\ell-1}}\otimes\mat{\field}{n_\ell}\otimes\mat{\field}{n_{\ell+1}},\\
\hfill A^{(1)}\otimes\cdots\otimes A^{(\ell)}\longmapsto
A^{(1)}\otimes\cdots\otimes A^{(\ell-1)}\otimes I_{n_\ell}\otimes
A^{(\ell)}\hfill
\end{multline*}
extended naturally to the mapping of matrices
\begin{multline*}
\rmat{\left(\mat{\field}{n_1}\otimes
\cdots\otimes\mat{\field}{n_{\ell-1}}\otimes\mat{\field}{n_{\ell+1}}\right)}{p}{q}\\
\hfill\longrightarrow \rmat{\left(\mat{\field}{n_1}\otimes
\cdots\otimes\mat{\field}{n_{\ell-1}}\otimes\mat{\field}{n_\ell}\otimes\mat{\field}{n_{\ell+1}}\right)}{p}{q}.
\end{multline*}
 We then have an
analogue of Theorem \ref{thm:rtriangle} as follows. Let $R$ be a
$p\times q$ matrix-valued noncommutative rational expression in
$\ell$ tuples of indeterminates, let
$Z^{(j)}\in\mattuple{\field}{n_j}{d}$, $j=1$, \ldots, $\ell-1$,
$Z\in\mattuple{\field}{n}{d}$, $Z'\in\mattuple{\field}{n'}{d}$, so
that
$$(Z^{(1)},\ldots,Z^{(\ell-1)},Z),\
(Z^{(1)},\ldots,Z^{(\ell-1)},Z')\in\dom R,$$ and let
$W\in\rmattuple{\field}{n}{n'}{d}$. Then
$$\left(Z^{(1)},\ldots,Z^{(\ell-1)},\begin{bmatrix}
Z & W\\
0 & Z'
\end{bmatrix}\right)\in\dom R$$
and
\begin{multline}\label{eq:rtriangle-ell}
P(n+n',pn_1\cdots
n_{\ell-1})R\left(Z^{(1)},\ldots,Z^{(\ell-1)},\begin{bmatrix}
Z & W\\
0 & Z'
\end{bmatrix}\right)P(n+n',qn_1\cdots
n_{\ell-1})^{\mtrans}\\
=\left[\begin{matrix} P(n,pn_1\cdots
n_{\ell-1})R(Z^{(1)},\ldots,Z^{(\ell-1)},Z)P(n,qn_1\cdots
n_{\ell-1})^{\mtrans}\\
0\end{matrix}\right.\hfill
\\
\left.\begin{matrix} P(n,pn_1\cdots
n_{\ell-1})\sum\limits_{j=1}^d\Delta_j(R)(Z^{(1)},\ldots,Z^{(\ell-1)},Z,Z')(W_j)P(n',qn_1\cdots
n_{\ell-1})^{\mtrans}\\
P(n',pn_1\cdots
n_{\ell-1})R(Z^{(1)},\ldots,Z^{(\ell-1)},Z')P(n',qn_1\cdots
n_{\ell-1})^{\mtrans}
\end{matrix}\right].
\end{multline}
Here
\begin{multline*}
\Delta_j(R)(Z^{(1)},\ldots,Z^{(\ell-1)},Z,Z')\in\rmat{\field}{p}{q}\otimes\mat{\field}{n_1}\otimes
\cdots\otimes\mat{\field}{n_{\ell-1}}\otimes\mat{\field}{n}\otimes\mat{\field}{n'}\\
\cong\rmat{\left(\mat{\field}{n_1}\otimes
\cdots\otimes\mat{\field}{n_{\ell-1}}\otimes\mat{\field}{n}\otimes\mat{\field}{n'}
\right)}{p}{q}
\end{multline*}
and
\begin{multline*}
\Delta_j(R)(Z^{(1)},\ldots,Z^{(\ell-1)},Z,Z')(W_j)\in
\rmat{\left(\mat{\field}{n_1}\otimes
\cdots\otimes\mat{\field}{n_{\ell-1}}\otimes\rmat{\field}{n}{n'}
\right)}{p}{q},
\end{multline*}
cf. the discussions preceding Theorem \ref{thm:rtriangle} and
following Example \ref{ex:realiz_dif}. The proof is analogous to
the proof of Theorem \ref{thm:rtriangle} except that instead of
establishing \eqref{eq:rtriangle} for polynomials we have to
establish \eqref{eq:rtriangle-ell} for pure tensors. Namely, we
have to show that if $R$ and $R'$ satisfy \eqref{eq:rtriangle-ell}
then so does $R\otimes R'$. This can be achieved using the
identity
\begin{multline*}
R(Z^{(1)},\ldots,Z^{(\ell-1)})\otimes R'\left(\begin{bmatrix}
Z & W\\
0 & Z'
\end{bmatrix}\right)\\
=\left(R(Z^{(1)},\ldots,Z^{(\ell-1)})\otimes
I\right)\left(I\otimes R'\left(\begin{bmatrix}
Z & W\\
0 & Z'
\end{bmatrix}\right)\right).
\end{multline*}
It follows from \eqref{eq:rtriangle-ell} that $\Delta_j$ preserves
the equivalence of matrix-valued noncommutative rational
expressions in $\ell$ tuples of indeterminates and can be thus
defined on matrices over $\skfield^{\otimes\ell}$. It also follows
that if $$(Z^{(1)},\ldots,Z^{(\ell-1)},Z),\
(Z^{(1)},\ldots,Z^{(\ell-1)},Z')\in\edom R,$$ then
$$\left(Z^{(1)},\ldots,Z^{(\ell-1)},Z,Z'\right)
\in\edom \Delta_j(R).$$
\end{rem}
\begin{rem}\label{rem:rdif-higher}
As a special case of the previous remark, we see that for a
matrix-valued noncommutative rational function $\mathfrak{R}$ and
for a word $w$ of length $\ell$, $\Delta^w\mathfrak{R}$ is the
matrix over $\skfield^{\otimes\ell}$ corresponding to the
equivalence class of $\Delta^w(R)$ for any $R\in\mathfrak{R}$. We
further conclude that $$\dom
\Delta^w\mathfrak{R}\supseteq(\dom\mathfrak{R})^{\ell+1},\quad
\edom \Delta^w\mathfrak{R}\supseteq(\edom\mathfrak{R})^{\ell+1}.$$
The second inclusion follows from the last statement of Remark
\ref{rem:rdifell}. The first inclusion follows from the equality
$\dom \Delta^wR=(\dom R)^{\ell+1}$ for a matrix-valued rational
expression $R$, which can be proved recursively. We leave the
details to the reader, noticing only that for the product we have
the following higher order Leibniz rule:
\begin{equation}\label{eq:Leibniz}
\Delta^w(R_1R_2)=\sum_{u,v\in\free_d\colon
w=uv}(\Delta^v(R_1)\otimes\underset{\ell-|v|\ {\rm
times}}{\underbrace{1\otimes\cdots\otimes
1}})\cdot(\underset{\ell-|u|\ {\rm
times}}{\underbrace{1\otimes\cdots\otimes
1}}\otimes\Delta^u(R_2)).
\end{equation}
\end{rem}

\subsection{Formal power series.}\label{s:FPS}
A matrix-valued noncommutative rational expression which is
regular at zero determines a noncommutative formal power series
with matrix coefficients. This correspondence is defined
recursively by inverting formal power series with invertible
constant term (the coefficient for $z^{\emptyword}$); see, e.g.,
\cite{BR}. Furthermore, $R_1$ and $R_2$ are equivalent if and only
if the corresponding formal power series coincide, so that the
noncommutative formal power series expansion of a matrix-valued
noncommutative rational function which is regular at zero is well
defined; see \cite[Remark 2.14]{KVV}. If
$\sum_{w\in\free_d}\mathfrak{R}_wz^w$ is the formal power series
expansion of $\mathfrak{R}$ then the formal power series expansion
of $\Delta_j\mathfrak{R}$ is given by
$$\sum_{u,v\in\free_d}\mathfrak{R}_{ug_jv}z^u\otimes z^{\prime v}.$$
For a proof, we use the recognizable series realization and
represent $\mathfrak{R}$ by a matrix-valued noncommutative
rational expression of the form \eqref{eq:rsr}. Therefore,
$\mathfrak{R}_w=CA^wB$. On the other hand, it follows from Example
\ref{ex:realiz} that the formal power series expansion of
$\Delta_j\mathfrak{R}$ is given by
$$\sum_{u,v\in\free_d}CA^uA_jA^vBz^u\otimes z^{\prime v}.$$ So, we see that the
coefficient for $z^u\otimes z^{\prime v}$ is exactly
$\mathfrak{R}_{ug_jv}$.

A similar argument using Example \ref{ex:realiz_dif} shows that
the formal power series expansion of $\Delta^w\mathfrak{R}$ for
$w=g_{i_\ell}\cdots g_{i_1}$ is given by
$$\sum_{u_1,\ldots,u_{\ell+1}\in\free_d}\mathfrak{R}_{u_1g_{i_1}u_2g_{i_2}\cdots
g_{i_\ell}u_{\ell+1}} (z^{(1)})^{u_1}\otimes
(z^{(2)})^{u_2}\otimes\cdots\otimes (z^{(\ell+1)})^{u_{\ell+1}}.$$
In particular, looking at the constant term of this expansion
(i.e., for $u_1=\ldots=u_{\ell+1}=\emptyset$), we see that
$$\Delta^w\mathfrak{R}(0,\ldots,0)=\mathfrak{R}_{w^{\mtrans}}.$$

\section{Conclusions}\label{s:conc}

Rational functions in noncommuting indeterminates occur in many
areas of system theory: most control problems involve rational
expressions in matrix parameters. In this paper we surveyed some
aspects of the theory of noncommutative rational functions, and
provided some pointers to a rapidly growing literature. We
discussed in some details a construction of the skew field of
noncommutative rational functions based on noncommutative rational
expressions and their matrix evaluations. We explained its role as
the universal field of fractions of the ring of noncommutative
polynomials. We gave an outline of a noncommutative realization
theory. Finally, we developed a difference-differential calculus
for noncommutative rational functions.

\section*{Aknowledgements}

It is our pleasure to thank two anonymous referees and Amnon
Yekutieli for useful comments and suggestions. We are also
grateful to Igor Klep for the \TeX macros used in the notation for
the free skew field.

\end{document}